\documentclass[letterpaper,11pt]{article}

\usepackage{amsmath, amsfonts, fullpage}
\usepackage{hyperref}
\usepackage{amsthm}
\usepackage[inline]{enumitem}
\usepackage[dvipsnames]{xcolor}
\usepackage{color}
\usepackage{graphics,graphicx}
\usepackage{mathtools}
\usepackage[totalheight=9.6in,totalwidth=6.15in,centering]{geometry}

\mathtoolsset{showmanualtags,showonlyrefs}
\hypersetup{pdfborder=0 0 .05}

\newcommand{\bigO}{\mathcal{O}}

\newcommand{\I}{{\rm i}}
\newcommand{\pp}{\mathbb{P}}

\newcommand{\ee}{\mathbb{E}}
\newcommand{\rr}{\mathbb{R}}
\newcommand{\nn}{\mathbb{N}}
\newcommand{\zz}{\mathbb{Z}}

\newcommand{\p}{\partial}
\newcommand{\de}{\delta}
\newcommand{\De}{\Delta}
\newcommand{\uno}[1]{\mathbf{1}_{#1}}
\newcommand{\ep}{\varepsilon}

\newcommand{\wt}{\widetilde}
\newcommand{\wh}{\widehat}

\newcommand{\qqand}{\qquad\text{and}\qquad}
\newcommand{\inv}[1]{\frac{1}{#1}}

\newcommand{\itwopii}[1]{\frac{1}{(2\pi\I)^{#1}}}

\renewcommand{\d}{\mathrm{d}}
\let\Re\relax
\DeclareMathOperator{\Re}{Re}

\newcommand{\eq}{\begin{equation}}
\newcommand{\eeq}{\end{equation}}

\newcommand{\equ}{{\rm eq}}
\newcommand{\Zeq}{Z^\equ}

\newcommand{\cz}{\mathcal{Z}}
\newcommand{\mcal}{\mathcal{M}}

\newcommand{\ts}{\hspace{0.1em}}
\newcommand{\tts}{\hspace{0.05em}}
\newcommand{\tsm}{\hspace{-0.1em}}

\newcommand{\aip}{\mathcal{A}}

\newcommand{\SM}{\mathcal{S}}
\newcommand{\SN}{\bar{\mathcal{S}}}
\renewcommand{\P}{\chi}
\newcommand{\fK}{\mathbf{K}}
\newcommand{\fI}{\mathbf{I}}

\newcommand{\fT}{\mathbf{S}}
\newcommand{\fA}{\mathbf{A}}
\newcommand{\fB}{\mathbf{B}}
\newcommand{\fD}{\mathbf{D}}
\newcommand{\fL}{\mathbf{L}}
\newcommand{\fR}{\mathbf{R}}
\newcommand{\fH}{H}
\newcommand{\tfB}{\widetilde{\mathbf{B}}}

\newcommand{\fh}{\mathfrak{h}}
\newcommand{\ff}{\mathfrak{f}}

\newcommand{\fX}{\mathfrak{X}}

\newcommand{\uptext}[1]{\text{\upshape{#1}}}
\newcommand{\<}{\langle\,}
\renewcommand{\>}{\,\rangle}
\DeclareMathOperator{\Ai}{Ai}

\DeclareMathOperator{\epi}{\uptext{epi}}
\DeclareMathOperator{\hypo}{\uptext{hypo}}
\DeclareMathOperator{\UC}{\uptext{UC}}

\newtheorem{prop}{Proposition}
\newtheorem{thm}{Theorem}

\newtheorem{cor}{Corollary}
\newtheorem{conj}{Conjecture}

\theoremstyle{remark}
\newtheorem{rmk}{Remark}

\newlength{\bibitemsep}
\newlength{\bibparskip}
\let\oldthebibliography\thebibliography
\renewcommand\thebibliography[1]{%
  \oldthebibliography{#1}%
  \setlength{\parskip}{\bibitemsep}%
  \setlength{\itemsep}{\bibparskip}%
}

\title{Airy process with wanderers, KPZ fluctuations, and a deformation of the Tracy--Widom GOE distribution}

\author{Karl Liechty \thanks{Department of Mathematical Sciences, DePaul University, Chicago, IL, 60614 USA \href{mailto:kliechty@depaul.edu}{\nolinkurl{kliechty@depaul.edu}}.} \and Gia Bao Nguyen \thanks{Department of Mathematics, KTH Royal Institute of Technology, SE-100 44 Stockholm, Sweden
\href{mailto:nguyengb@kth.se}{\nolinkurl{nguyengb@kth.se}}.}
\and Daniel Remenik \thanks{Departamento de Ingenier\'ia Matem\'atica and Centro de Modelamiento Matem\'atico (UMI-CNRS 2807), Universidad de Chile \href{mailto:dremenik@dim.uchile.cl}{\nolinkurl{dremenik@dim.uchile.cl}}.}}

\begin{document}
\maketitle

\begin{abstract}
We study the distribution of the supremum of the Airy process with $m$ wanderers minus a parabola, or equivalently the limit of the rescaled maximal height of a system of $N$ non-intersecting Brownian bridges as $N\to\infty$, where the first $N-m$ paths start and end at the origin and the remaining $m$ go between arbitrary positions.
The distribution provides a $2m$-parameter deformation of the Tracy--Widom GOE distribution, which is recovered in the limit corresponding to all Brownian paths starting and ending at the origin.

We provide several descriptions of this distribution function: (i) A Fredholm determinant formula; (ii) A formula in terms of Painlev\'e II functions; (iii) A representation as a marginal of the KPZ fixed point with initial data given as the top path in a stationary system of reflected Brownian motions with drift; (iv) A characterization as the solution of a version of the Bloemendal--Virag PDE \cite{Bloemendal-Virag11, Bloemendal-Virag16} for spiked Tracy--Widom distributions; \linebreak (v) A representation as a solution of the KdV equation.
We also discuss connections with a model of last passage percolation with boundary sources.
\end{abstract}

\section{Introduction}

The Tracy--Widom (TW) GOE, GUE, and GSE distributions describe the rescaled location of the largest eigenvalue in self-adjoint random matrix ensembles with real, complex, and quaternion entries, respectively, as the size of the matrix approaches infinity.  Each of these distributions may be described in terms of a certain solution to the Painlev\'{e} II equation, or equivalently in terms of a Fredholm determinant or determinants of an operator involving Airy functions. The Painlev\'{e} II equation (PII) is the second order nonlinear ODE
\eq\label{eq:PII}
q''(s) = 2q(s)^3+sq(s),
\eeq
and we consider the {\it Hastings--McLeod solution} to this equation, which behaves as 
\eq\label{eq:HM}
q(s) \sim \Ai(s), \qquad  \textrm{ as }\qquad s\to+\infty,
\eeq
where $\Ai$ is the Airy function.
 Define the functions
\eq\label{eq:PIIfunctions}
v(s) = \int_s^\infty q(x)^2\,dx, \quad E(s) = \exp\!\left(\tsm-\frac{1}{2} \int_s^\infty q(x)\,dx\right), \quad F(s) = \exp\!\left(\tsm-\frac{1}{2} \int_s^\infty v(x)\,dx\right).
\eeq
Then the distributions for the Tracy--Widom GOE, GUE, and GSE distribution functions, denoted as $F_1$, $F_2$, and $F_4$, respectively, are given as 
\eq\label{eq:TWdistr1}
F_1(s) = F(s)E(s), \qquad F_2(s) = F(s)^2, \qquad F_4(2^{-2/3}s) = \frac{1}{2}\!\left(E(s)+\frac{1}{E(s)}\right)F(s).
\eeq
For the description in terms of Fredholm determinant(s), let $\fB_s$ be the integral operator acting in $L^2([0,\infty))$ with the kernel $\fB_s(x,y)$ defined by
\eq\label{eq:Bs}
\fB_s(x,y) \coloneqq \Ai(x+y+s).
\eeq
Then the Tracy--Widom distribution functions are also given as
\eq\label{eq:TWdistr2}
F_1(s) = \det(\fI-\fB_s), \quad F_2(s) =  \det(\fI-\fB_s^2), \quad F_4(2^{-2/3}s) = \frac{1}{2}\left( \det(\fI-\fB_s)+ \det(\fI+\fB_s)\right),
\eeq
with all determinants in $L^2([0,\infty))$.

The Tracy--Widom distributions were originally discovered in the context of random matrix theory \cite{Tracy-Widom94, Tracy-Widom96}, but have since become ubiquitous in random systems with a high degree of correlation. In particular, the Tracy--Widom GUE and GOE distributions appear as one-point distributions in random growth models in the Kardar--Parisi--Zhang (KPZ) universality class with narrow-wedge (GUE) and flat (GOE) initial conditions; the GSE case arises similarly in half-space KPZ models. For narrow-wedge initial data the multi-point fluctuations for KPZ models are described by the {\it Airy$_2$ process}. This coincides with the limiting process at the edge of Dyson Brownian motion for complex Hermitian matrices, a very natural dynamic version of the Gaussian Unitary Ensemble (GUE).

Beyond the Tracy--Widom GO/U/SE distributions and the Airy$_2$ process, the relationship between random matrix theory and KPZ models becomes a bit more tenuous. One might expect, for instance, that the multi-point fluctuations for KPZ models with flat initial data coincide with the limiting process at the edge of Dyson Brownian motion for real symmetric matrices, but this is not so \cite{Bornemann-Ferrari-Prahofer08}. The former process, known as the {\it Airy$_1$ process}, is a central object in KPZ models, which has TW-GOE marginals but does not seem to appear at all in random matrix models.
Instead, the TW-GOE distribution is connected to the Airy$_2$ process through the well known formula \cite{Corwin-Quastel-Remenik13,Johansson03}
\eq\label{eq:F1_identity}
\pp\!\left(\sup_{t\in \rr} (\aip(t) - t^2) \le x\right) = F_1(2^{2/3} x),
\eeq
where $\aip$ is the Airy$_2$ process, a particular case of a general variational formula satisfied by KPZ models \cite{Corwin-Liu-Wang16,Dauvergne-Ortmann-Virag18,Matetski-Quastel-Remenik17,Quastel-Remenik14}. This identity implies that the maximum of Dyson's Brownian motion minus a parabola, or equivalently an ensemble of non-intersecting Brownian bridges, is described by the Tracy--Widom GOE distribution in the appropriate scaling limit.

In this paper we introduce a natural deformation of the TW-GOE distribution which appears both in the setting of random matrices/non-intersecting Brownian motions and as a scaling limit of KPZ models. On the random matrix side we arrive at our deformation by deforming the left hand side of \eqref{eq:F1_identity}, replacing the Airy$_2$ process with the {\it Airy process with wanderers}; this is a deformation of the Airy$_2$ process introduced in \cite{Adler-Ferrari-van_Moerbeke10} which arises as the scaling limit of systems $N$ non-intersecting Brownian bridges in which most particles are conditioned to return start at and return to a common point, but a few are conditioned to start and end elsewhere. 

On the KPZ side this deformed TW-GOE distribution will arise as a marginal of the KPZ fixed point (the universal process, constructed in \cite{Matetski-Quastel-Remenik17}, conjectured to govern the asymptotic fluctuations of all models in the class) with initial data constructed as a certain random deformation of the flat initial conditions; concretely, one may think of it as coming from the totally asymmetric simple exclusion process, one of the paradigmatic models in the KPZ class, started from a random perturbation of the periodic initial conditions.
The two perspectives lead to several alternative descriptions of our distribution: as the Fredholm determinant of a finite rank perturbation of the kernel $\fB_s$ from \eqref{eq:Bs}, as a formula given in terms of a Lax pair for the Hastings--McLeod solution of PII, and as a solution of two different PDEs.

We remark that the Airy process with wanderers also appears as the scaling limit of the last passage times in a certain directed last passage percolation (LPP) model with boundary sources \cite{Borodin-Peche08}.
As a consequence, the supremum of the Airy process with wanderers minus a parabola, which is the primary object studied in this paper, is related to the point-to-line last passage time for this model.
The LPP perspective is discussed in Section \ref{sec:cts}, where we also comment on connections with recent results obtained in \cite{FitzGerald-Warren20}.

\section{Main results}

\subsection{Non-intersecting Brownian bridges and Airy processes with wanderers}\label{sec:nonint}

Consider $N$ non-intersecting Brownian bridges $(B_1(t),\dotsc,B_N(t))$ on the time interval $[-1,1]$, labeled so that $B_1(t)\leq B_2(t)\leq\dotsm\leq B_N(t)$, such that the first $N-m$ paths start and end at $0$, while the $m$ remaining top {\it outlier paths} go from a set of $m$ locations $\alpha_1\geq\dotsm\geq\alpha_m\geq0$ to another set of locations $\beta_1\geq\dotsm\geq\beta_m\geq0$.
When $\alpha_i=\beta_i=0$ for $i=1,\dotsc,m$, the system has a limit shape bounded by the ellipse $\mathcal{C}\coloneqq\{(t,\pm\sqrt{2N(1-t^2)})\!:t\in[-1,1]\}$ and the fluctuations around the top boundary are described by the Airy line ensemble minus a parabolic shift, the top line of which is precisely the Airy$_2$ process $\aip$ which we mentioned in the introduction. In fact, since the top paths reach a maximal height of about $\sqrt{2N}$ due to the conditioning on non-intersection, the same fluctuation process arises whenever the $\alpha_i$'s and $\beta_i$'s stay bounded as $N\to\infty$.

In \cite{Adler-Ferrari-van_Moerbeke10} it was shown that a new fluctuation process, the \emph{Airy process with $m$ wanderers}, arises if one scales the starting and ending locations of the top $m$ non-intersecting Brownian bridges critically around $\sqrt{2N}$.
The correct scaling corresponds to choosing
\begin{equation}\label{eq:ab-scaling}
\alpha_i=\sqrt{2N}\!\left(1-\frac{a_i}{N^{1/3}}\right)\qqand\beta_i=\sqrt{2N}\!\left(1-\frac{b_i}{N^{1/3}}\right)
\end{equation}
for fixed choices of $a_1\leq a_2\leq\dotsm\leq a_m$ and $b_1\leq b_2\leq\dotsm\leq b_m$ satisfying additionally $-a_1\leq b_1$; for later convenience, we have changed the sign of $a_i$ in our scaling compared to \cite{Adler-Ferrari-van_Moerbeke10}.
The Airy process with wanderers, which we denote as $\aip^{(\vec a,\vec b)}_m$, then occurs at a spatial scale of order $x-\sqrt{2N}=\bigO(N^{-1/6})$, which is much smaller than the scaling of the starting/ending points \eqref{eq:ab-scaling}, which is $\alpha_i-\sqrt{2N} = \bigO(N^{1/6})$.
Thus the condition $-a_1<b_1$ ensures that the top path, and thus all the $m$ wanderers, interact  non-trivially with the cloud of $N-m$ particles, giving rise to the new fluctuation process; otherwise, the wanderers don't feel the effect of the bulk, and the fluctuations become Gaussian, see Figure \ref{fig:wand} below as well as Figure 1 in \cite{Adler-Ferrari-van_Moerbeke10}.
The precise statement is the following: under the scaling \eqref{eq:ab-scaling} with $a_1\leq a_2\leq\dotsm\leq a_m$ and $b_1\leq b_2\leq\dotsm\leq b_m$ and $-a_1<b_1$, the top Brownian path $(B_N(t))_{t\in[-1,1]}$ in the system introduced above satisfies
\begin{equation}\label{eq:wm-to-aw}
\sqrt{2}N^{1/6}\left(B_N(N^{-1/3}t)-\sqrt{2N}\right)\xrightarrow[N\to\infty]{}\aip^{(\vec a,\vec b)}_m(t)
\end{equation}
in distribution, uniformly over $t$ in compact subsets of $\rr$.
The convergence of the finite-dimensional distributions follows from \cite[Thm. 1.2]{Adler-Ferrari-van_Moerbeke10}, which in fact proves much more, namely the joint convergence of the gap probabilities for any fixed number of paths at the top of the system to those of a limiting (extended) point process\footnote{While this limiting point process is what \cite{Adler-Ferrari-van_Moerbeke10} calls the Airy process with wanderers, throughout this paper we use always use this term mean the scaling limit of the top path; this is similar  to how the Airy$_2$ process names both the scaling limit of the top path in Dyson Brownian motion and the determinantal point process arising from scaling the GUE eigenvalues at the edge.
}.
The upgrade to uniform convergence on compact sets, and in fact the existence of a continuous process corresponding to the scaling limit of the top path, follows from the results of \cite{Corwin-Hammond11}.

The limiting process $\aip^{(\vec{a},\vec{b})}_{m}$ can be defined through its finite dimensional distributions\footnote{Recall we have changed the sign of the $a_i$'s compared to \cite{Adler-Ferrari-van_Moerbeke10}.}: for $t_1<t_2<\ldots<t_n$,
\begin{equation}\label{eq:awextdeter}
\pp\!\left(\aip^{(\vec{a},\vec{b})}_{m}(t_j)\leq r_j,\,j=1,\dotsc,n\right)=\det\!\left(\fI-\P_r\fK^{(\tilde{a},\tilde{b})}_{m}\P_r\right)_{L^2(\{t_1,\ldots,t_n\}\times\rr)}
\end{equation}
where for a fixed vector $r\in\rr^n$ we set
\begin{equation}\label{eq:defChis}
\P_r(t_j,x)=\uno{x>r_j}\qqand\bar\P_r(t_j,x)=\uno{x\leq r_j},
\end{equation}
which we also regard as multiplication operators acting on the space $L^2(\{t_1,\dotsc,t_m\}\times\rr)$, and where the extended kernel $\fK^{(\tilde{a},\tilde{b})}_{m}$ is defined as
\begin{align}\label{eq:awextkernel}
 \fK^{(\vec{a},\vec{b})}_{m}&(s,x;t,y)=-\tfrac{1}{\sqrt{4\pi(t-s)}}e^{-(y-x)^2/(4(t-s))-(t-s)(x+y)/2+(t-s)^3/12}\uno{s<t}\\
 &+\frac{1}{(2\pi\I)^2}\int_{\Gamma_{\<\vec b-s}}du\int_{\Gamma_{-\vec a-t\>}}dv \frac{e^{u^3/3-xu}}{e^{v^3/3-yv}}\frac{1}{(u+s)-(v+t)}\prod_{k=1}^m\frac{u+a_k+s}{v+a_k+t}\frac{v-b_k+t}{u-b_k+s}.
\end{align}
The integration contours are as follows: $\Gamma_{\<\vec{b}-s}$ goes from $e^{-\pi\I/3}\infty$ to $e^{\pi\I/3}\infty$ and passes to the left of each ${b}_i-s$, $\Gamma_{\vec{a}-t\>}$ goes from $e^{-2\pi\I/3}\infty$ to $e^{2\pi\I/3}\infty$ and passes to the right of each $a_i-t$, and they are such that the shifted contours $\Gamma_{\<\vec{b}-s}+s$ and $\Gamma_{\vec{a}-t\>}+t$ do not intersect.

Note finally the kernel $\fK^{(\tilde{a},\tilde{b})}_{m}$ makes perfect sense if some of the $a_i$'s or the $b_i$'s are set to $\infty$: all that happens is that the corresponding factors in the integrand on the right hand side of \eqref{eq:awextdeter} disappear.
Physically, in view of the scaling \eqref{eq:ab-scaling}, one expects that this should recover the case where the corresponding $\alpha_i$'s or $\beta_i$'s equal $0$.
This is indeed the case, and can be derived without additional difficulty by repeating the arguments of \cite{Adler-Ferrari-van_Moerbeke10} with the endpoints of the corresponding Brownian paths tied to the origin.
In particular, taking both $a_m$ and $b_m$ to $\infty$ in $\aip^{((a_1,\dotsc,a_m),(b_1,\dotsc,b_m))}_m$ one simply recovers the Airy process with $m-1$ wanderers $\aip^{((a_1,\dotsc,a_{m-1}),(b_1,\dotsc,b_{m-1}))}_{m-1}$, while if all $a_i$'s are taken to $\infty$ and all $b_i$'s are set at a common location $b\in(0,\infty)$ one recovers the particular case studied in \cite{Adler-Delepine-van_Moerbeke09}.

\begin{figure}
\centering
\footnotesize
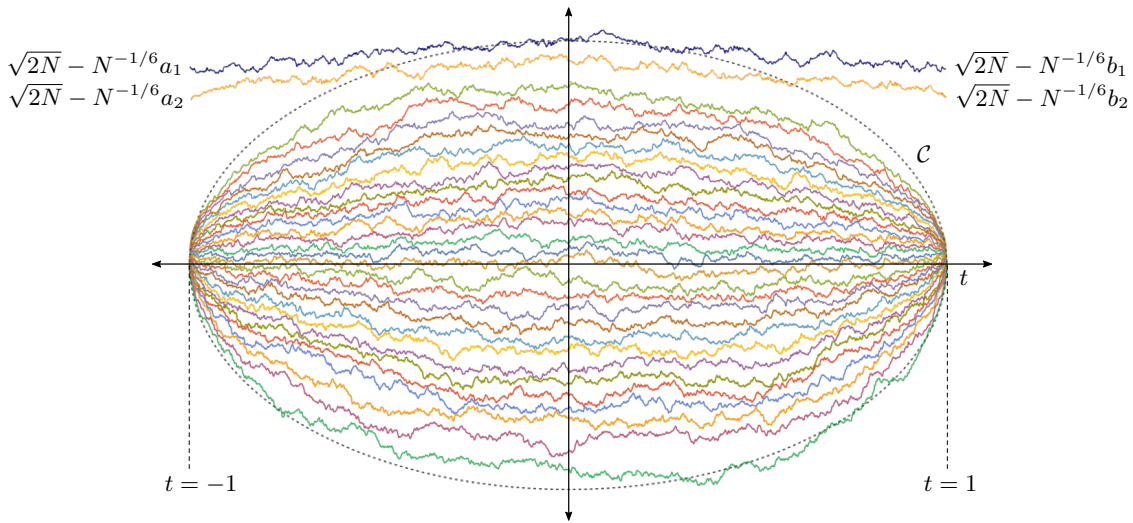
\caption{\small A system of non-intersecting Brownian bridges with $2$ outliers going from $(\alpha_1,\alpha_2)$ to $(\beta_1,\beta_2)$ with scalings as in \eqref{eq:ab-scaling}.
Here the $a_i$'s and $b_i$'s are positive (so the condition $-a_1<b_1$ holds) and thus the outlier paths enter the region defined by the ellipse $\mathcal{C}$ and hence they interact nontrivially with the paths in the bulk.}
\label{fig:wand}
\end{figure}

The primary object of interest in the current paper is the distribution of the supremum of the Airy process with wanderers minus a parabola:
\eq\label{eq:supwm-to-supai}
\begin{aligned}
F^{(\vec a,\vec b)}_m(r)&=\pp\Big(\sup_{t\in \rr}\big(\aip^{(\vec a,\vec b)}_{m}(t)-t^2\big)\leq r\Big)\\
&=\lim_{N\to\infty}\pp\Big(\sqrt{2}N^{1/6}\!\left({\textstyle\sup_{t\in[-1,1]}}B_N(t)-\sqrt{2N}\right)\leq r\Big),
\end{aligned}
\eeq
where the second equality comes from \eqref{eq:wm-to-aw} (which only gives convergence over $t$ in compact sets, but from the arguments of \cite{Corwin-Hammond11}, see Corollary 5.2 there, it follows that the probability that the maximum is attained outside of a given box $[-M,M]$ goes to $0$ as $M\to\infty$)\footnote{Alternatively, our proof of \eqref{eq:Fab-det} can be repeated for the maximal height of a finite system of non-intersecting Brownian bridges as in \cite{Nguyen-Remenik17}, and the resulting Fredholm determinant can be checked to converge to the right hand side of \eqref{eq:Fab-det} in the right scaling.}. 
In terms of LPP with boundary sources as in \cite{Borodin-Peche08}, one expects that $F^{(\vec a,\vec b)}_m$ is the distribution of the asymptotic fluctuations of the corresponding point-to-line last passage times. 

From the preceding discussion one should expect that 
\begin{equation}\label{eq:limits1}
\lim_{a_1,b_1,\dotsc,a_m,b_m\to +\infty}F^{(\vec a,\vec b)}_m(r)=F_1(2^{2/3}r).
\end{equation}
On the other hand, in view of \eqref{eq:supwm-to-supai} and the scaling introduced above, one should expect that $\sup_{t\in \rr}\big(\aip^{(\vec a,\vec b)}_{m}(t)-t^2\big)=\infty$ if any of the $a_i$'s or $b_i$'s are negative, and in particular that
\begin{equation}\label{eq:limits2}
\lim_{a_1\to 0^+ \ \textrm{or} \ b_1 \to 0^+}F^{(\vec a,\vec b)}_m(r)=0.
\end{equation}
In fact, if that is the case then one (or both) of the endpoints of the top path lies at a location greater than or equal than $\sqrt{2N}$, and the supremum of the path near that endpoint does not feel the bulk and has fluctuations around $\sqrt{2N}$ which are of order $1$.
In view of this, we will assume in the sequel that all $a_i$'s and $b_i$'s are positive (which in particular implies the required condition $-a_1<b_1$).

We will prove in Section \ref{sec:limittrans} that both limits \eqref{eq:limits1} and \eqref{eq:limits2} hold. 

\subsection{Fredholm determinant formula for the maximal height}\label{sec:det}

Our first result is an explicit Fredholm determinant formula for $F^{(\vec a,\vec b)}_m$.

\begin{thm} \label{thm:DetFormula}
Consider the Airy process with $m$ wanderers $\aip^{(\vec{a},\vec{b})}_{m}$ with parameters satisfying $0<a_1\leq\dotsm\leq a_m$ and $0<b_1\leq\dotsm\leq b_m$.
Then
\eq\label{eq:Fab-det}
F^{(\vec a,\vec b)}_m(r) = \det(\fI-\P_0\tfB_{2^{2/3}r}^{(\vec{a},\vec{b})}\P_0)_{L^2(\rr)},
\eeq
where we recall, $\P_0(x)=\uno{x>0}$, and 
where $\tfB_{2^{2/3}r}^{(\vec{a},\vec{b})}$ is the integral operator with kernel
\begin{equation}\label{eq:Bab}
\tfB^{(\vec a,\vec b)}_{2^{2/3}r}(x,y)=\frac1{2\pi\I}\int_{\langle}dw\,e^{2w^3/3-(x+y+2r)w}\prod_{k=1}^m\frac{a_k+w}{a_k-w}\frac{b_k+w}{b_k-w}.
\end{equation}
Here the contour $\langle$ goes from $e^{-\pi\I/3}\infty$ to $e^{\pi\I/3}\infty$ and passes to the left of all $a_i$'s and all $b_i$'s.
\end{thm}

The result is proved in Section \ref{sec:DetFormula}. 
The right hand side of \eqref{eq:Fab-det} can also be written as the Fredholm determinant of a finite rank perturbation of a scaled version of the kernel $\fB_{2^{2/3}r}$ from \eqref{eq:Bs}, see Proposition \ref{prop:finiterank-B} in Section \ref{sec:PII} below.

From the symmetry of the kernel in \eqref{eq:Bab} it follows directly that the supremum of the Airy process with $m$ wanderers with parameters $0<a_1\leq\dotsm\leq a_m$ and $0<b_1\leq\dotsm\leq b_m$ has the same distribution as the one with $2m$ wanderers, with starting points set to $\infty$ and the new endpoints given by the $a_i$'s and the $b_i$'s:

\begin{cor}
For any $0<a_1\leq\dotsm\leq a_m$ and $0<b_1\leq\dotsm\leq b_m$,
\begin{equation}\label{eq:absymm}
F^{(\vec a,\vec b)}_m(r)=F^{(\infty,\langle \vec a,\vec b\rangle)}_{2m}(r)
\end{equation}
with $\langle \vec a,\vec b\rangle$ representing the non-decreasing ordering of the $2m$ parameters.
\end{cor}
While immediate from Theorem \ref{thm:DetFormula}, this fact is far from obvious from the perspective of its physical motivation in terms of non-intersecting Brownian bridges (or in terms of LPP with boundary sources, see Section \ref{sec:cts}).

As expected, when all the $a_i$ or $b_i$ parameters are set to $\infty$ we recover the case where all Brownian bridges start and end at the origin.
In fact, the kernel in \eqref{eq:Bab} becomes 
\begin{align}\label{eq:Btildeinfty}
\tfB_{2^{2/3}r}(x,y)&\coloneqq\tfB^{(\infty,\infty)}_{2^{2/3}r}(x,y)=2^{-1/3}\frac1{2\pi\I}\int_{\langle}d w\,e^{2w^3/3-(x+y+2r)w}\\
&=2^{-1/3}\Ai(2^{-1/3}(x+y+2r))=2^{-1/3}\fB_{2^{2/3}r}(2^{-1/3}x,2^{-1/3}y),
\end{align}
where we used the contour integral formula for the Airy function 
\begin{equation}
 \Ai(x)=\frac1{2\pi\I}\int_{\langle}d w\,e^{w^3/3-xw}.\label{eq:airy}
\end{equation}
After rescaling the kernel the right hand side of \eqref{eq:Fab-det} becomes $\det(\fI-\P_0\fB_{2^{2/3}r}\P_0)$, which in view of \eqref{eq:TWdistr2}, matches \eqref{eq:F1_identity}.

Taking the parameters to $0$ in the Fredholm determinant formula is much subtler.
For simplicity, consider the case $m=1$ with $a_1=\infty$ and $b_1=b$.
The kernel in this case becomes
\[\tfB^{(\infty,b)}_{2^{2/3}r}(x,y)=\frac1{2\pi\I}\int_{\langle}dw\,e^{2w^3/3-(x+y+2r)w}\frac{b+w}{b-w}\]
where we recall the contour passes to the left of $b$.
One may be tempted to simply take $b\to0$ inside the kernel, but this leads to $-2^{-1/3}\Ai(2^{-1/3}(x+y+2r))$ which does not give the right answer\footnote {In fact, it leads to $\det(\fI+\P_0\fB_{2^{2/3}r}\P_0)=F_2(2^{2/3} r)/F_1(2^{2/3}r)$, which is not even a distribution function!}. The problem is that the convergence of the kernel holds pointwise, but not in trace norm.
To get the right answer we proceed formally as follows.
First we deform the contour $\langle$ to the imaginary axis and rescale $x\longmapsto x/b$ and $y\longmapsto y/b$ in the Fredholm determinant and $w\longmapsto bw$ in the contour integral appearing inside the kernel to get the formula $F^{(\infty,b)}_1(r)=\det(\fI-\P_0\widehat\fB_b\P_0)$ with $\widehat\fB_b(x,y)=\frac1{2\pi\I}\int_{\I\rr}\ts dw\,e^{2b^3w^3/3-(x+y)w}\frac{1+w}{1-w}$.
Taking $b\to0$ inside the kernel we get $\widehat\fB_0(x,y)=\frac1{2\pi\I}\int_{\I\rr}\ts dw\,e^{-(x+y)w}\frac{1+w}{1-w}$.
Now let $f(z)=e^{-z}$.
An easy computation shows that $\widehat\fB_0f(z)=\frac1{2\pi\I}\int_{\I\rr}\ts dw\,e^{-zw}\frac{1}{1-w}$, and a Fourier transform calculation shows that the right hand side equals $e^{-z}\uno{z\geq0}$.
Hence $f$ is an eigenfunction of $\widehat\fB_0$ with eigenvalue $1$, so that that $F^{(\infty,0)}_1(r)=\det(\fI-\P_0\widehat\fB_0\P_0)=0$, matching \eqref{eq:limits2}.

\subsection{Painlev\'{e} II formula}

We turn next to the PII expression for $F^{(\vec a,\vec b)}_m$.
In addition to the Hastings--McLeod solution to PII denoted $q$ and defined in \eqref{eq:PII} and \eqref{eq:HM}, we need to introduce two special functions, $f(x,w)$ and $g(x,w)$, defined as certain solutions to the following linear differential equations:
\begin{align}
\frac{\p}{\p w}\tsm\begin{pmatrix} f \\ g\end{pmatrix}& = \begin{pmatrix} q(x)^2 & -wq(x)-q'(x) \\ -wq(x)+q'(x) & w^2 - x - q(x)^2 \end{pmatrix}\begin{pmatrix} f \\ g\end{pmatrix} \label{eq:Lax1},\\
\frac{\p}{\p x}\tsm\begin{pmatrix} f \\ g\end{pmatrix}& = \begin{pmatrix} 0 & q(x) \\ q(x) & -w \end{pmatrix}\begin{pmatrix} f \\ g\end{pmatrix}. \label{eq:Lax2}
\end{align}
These equations comprise a  {\it Lax pair} for PII, meaning that the compatibility of the above equations implies that $q(x)$ satisfies PII. We let $f(x,w)$ and $g(x,w)$ be the solution to \eqref{eq:Lax1} satisfying the initial condition (where $F(x)$ was defined in \eqref{eq:PIIfunctions}) 
\eq\label{eq:IC}
\begin{pmatrix} f(x,0) \\ g(x,0)\end{pmatrix} = F(x)\!\begin{pmatrix} 1 \\ 1\end{pmatrix}.
\eeq

We present the PII formulas for $F^{(\vec a,\vec b)}_m$ only for the case $a_1=\dots=a_m=\infty$; the formula for arbitrary $a_i$'s is identical in light of the symmetry \eqref{eq:absymm} between the $a_i$'s and the $b_i$'s.

\begin{thm}\label{thm:PII}
Consider the Airy process with $m$ wanderers $\aip^{(\infty,\vec{b})}_{m}$ with $0<b_1<\dotsm<b_m$. 
We have the formulas
\begin{align}
F^{(\infty,\vec b)}_m(r) &=\frac{ F_1(2^{2/3} r)}{\prod_{1\le j<k\le m} (b_k-b_j)} \det\!\left[\left(b_j+D_r\right)^{k-1}\left(f(2^{2/3} r, 2^{1/3} b_j) - g(2^{2/3} r, 2^{1/3} b_j)\right)\right]_{j,k=1}^m\quad\label{eq:PII1}\\
&=\frac{ F_1(2^{2/3} r)}{\prod_{1\le j<k\le m} (b_k-b_j)} \det\!\left[b_j^{k-1}\left(f(2^{2/3} r, 2^{1/3} b_j) +(-1)^{k} g(2^{2/3} r, 2^{1/3} b_j)\right)\right]_{j,k=1}^m,\label{eq:PII2}
\end{align}
where $D_r$ denotes partial derivative with respect to $r$.
When some of the $b_i$'s coincide the two formulas still hold after using l'H\^opital's rule to compute the right hand side as a limit.
\end{thm}

Our proof of formula \eqref{eq:PII1} will be based on direct manipulations of the Fredholm determinant formula from Theorem \ref{thm:DetFormula}, and is presented in Section \ref{sec:PII}. 
Then \eqref{eq:PII2} follows from \eqref{eq:PII1} by using \eqref{eq:Lax2} and elementary row operations.
For $m=1$ and $m=2$, there is an alternative derivation based on using the Karlin--McGregor formula for non-intersecting Brownian excursions and discrete orthogonal polynomials to derive an explicit formula for their maximal height.
This second approach is the one that originally led us to these formulas; we sketch it in Appendix \ref{BEandOP}.

For the special case $m=1$ \eqref{eq:PII1} becomes
\begin{equation}\label{def_dist}
F^{(\infty,b)}_m(r)  = F_1(2^{2/3} r)(f(2^{2/3} r, 2^{1/3} b) - g(2^{2/3} r, 2^{1/3} b)).
\end{equation}
Since $f(x,w)\to 1$ and $g(x,w)\to 0$ as $w\to\infty$ \cite{Baik-Rains00,Baik-Rains01a}, it is clear that the right hand side of \eqref{def_dist} approaches $F_1(2^{2/3} r)$ as $b\to\infty$ for fixed $r\in\rr$ while, since $f(x,0) = g(x,0)$, the right hand side of \eqref{def_dist} vanishes for $b=0$.
This recovers the expected behavior stated in \eqref{eq:limits1} and \eqref{eq:limits2}.
While one could attempt to generalize the argument to $m>1$, we will generalize the proof instead using the connection with the KPZ fixed point discussed next (see Section \ref{sec:limittrans}).

\subsection{Connection with KPZ fluctuations}\label{sec:KPZfluct}

The one dimensional KPZ universality class consists of a broad collection of random growth models, last passage percolation and directed polymers, and random stirred fluids.
The name of the class comes from the Kardar--Parisi--Zhang SPDE $\partial_t h = \lambda(\partial_xh)^2  + \nu \partial_x^2h + \sigma \xi$ with $\xi$ a space-time white noise, a canonical continuum equation for random growth introduced in \cite{Kardar-Parisi-Zhang86}.
An analogue of a height function $h(t,x)$ can be associated to every model in the class, and the main goal of the subject is to study the long time, large scale fluctuations of $h$.
The conjecture is that for every model in the KPZ universality class the height function converges to a universal limit $\fh(t,x)$ under the 1:2:3 scaling corresponding to letting $\ep\to0$ in
\begin{equation}\label{eq:123sc}
c_1\ep^{1/2} h(c_2\ep^{-3/2} t, c_3\ep^{-1} x) - C_\ep t,
\end{equation}
for some model-dependent constants $c_1$, $c_2$, $c_3$ and $C_\ep$.
This universal process $\fh(t,x)$ is known as the \emph{KPZ fixed point}, and was constructed in \cite{Matetski-Quastel-Remenik17} as the limit of the 1:2:3 rescaled height function for a specific model in the class, the totally asymmetric simple exclusion process (TASEP); later work shows that it arises too from other models related to TASEP \cite{Dauvergne-Ortmann-Virag18,Matetski-Quastel-Remenik20,Nica-Quastel-Remenik19,Nica-Quastel-Remenik20} as well as from the KPZ equation itself \cite{Quastel-Sarkar20,Virag20}.
The Airy$_2$ and Airy$_1$ processes mentioned in the introduction correspond to the KPZ fixed point at time $t=1$ in the case of two special choices of initial data: narrow-wedge (meaning $\fh(0,x)$ equal to $0$ for $x=0$ and $-\infty$ for all other $x$) for Airy$_2$; and flat (meaning $\fh(0,x)=0$ for all $x$) for Airy$_1$, see \cite{Matetski-Quastel-Remenik17}.

To be more precise, and since it will play a role in our proofs, let us introduce briefly the \emph{TASEP height function} $(h_t(x))_{x\in\zz}$ and its convergence to the KPZ fixed point.
For each fixed $t$ the height function $h_t$ is a simple random walk path, i.e. $h_t(x)-h_t(x-1)\in\{-1,1\}$ for each $x\in\zz$; the global height is fixed by imposing $h_0(0)=0$.
The dynamics of the TASEP height function is that local maxima become local minima independently at rate 1; i.e. if $h_t(x)=h_t(x\pm1)+1$ then the transition $h_t(x)\mapsto h_t(x)-2$ occurs at rate $1$ independently for different $x$'s, the rest of the height function remaining unchanged.
The \emph{TASEP particle system} is simply the discrete derivative of the height function: letting $\hat\eta_t(x)=h_t(x+1)-h_t(x)$ and thinking of $\hat\eta_t(x)=1$ as there being a particle at $x$ at time $t$ and $\hat\eta_t(x)=-1$ as the site being empty, the above dynamics correspond to particles jumping to the right at rate $1$ independently but subject to the exclusion rule that jumps onto occupied sites are forbidden.

Introduce the space $\UC$ of upper semi-continuous functions $\fh\!:\rr\longrightarrow\rr\cup\{-\infty\}$ satisfying $\fh(x)\le A|x| + B$ for some $A,B<\infty$, with the topology of local Hausdorff convergence of hypographs.
It was proved in \cite{Matetski-Quastel-Remenik17} that if $\fh_0$ is a possibly random element of $\UC$ and $\ep^{1/2}h_0(2\ep^{-1}x)\longrightarrow\fh_0(x)$ as $\ep\to0$, then for the TASEP height function it holds that
\begin{equation}\label{eq:TASEPtoFP}
\ep^{1/2}\big(h_{2\ep^{-3/2}t}(2\ep^{-1}x)+\ep^{-3/2}t\big)\xrightarrow[\ep\to0]{}\fh(t,x;\fh_0),
\end{equation}
all in distribution in $\UC$.
The limit $\big(\fh(t,\cdot;\fh_0)\big)_{t\geq0}$ evolves as a Markov process taking values in $\UC$, and $\fh(t,x;\fh_0)$ denotes the state of the process at time $t$ given its initial state $\fh_0$.
The transition probabilities for this Markov process can be expressed through a Fredholm determinant formula, see \cite[Defn. 3.12]{Matetski-Quastel-Remenik17}.

The main result presented in this section relates the maximal height of the Airy process with wanderers minus a parabola with the distribution of the KPZ fixed point at time $t=1$ for a particular choice of initial data, which we describe next.
In view of \eqref{eq:absymm} again, we state everything in terms of the case where all $a_i$'s equal infinity.
Fix $m\geq1$ and $0<b_1\leq\dotsm\leq b_m$, and introduce a system of $m$ reflected Brownian motions with drift with a wall at the origin $0\leq Z^0_1(t)\leq Z^0_2(t)\leq\dotsm\leq Z^0_m(t)$ (which we will simply refer to as RBMs) as follows. 
Write $Z^0_0(t)=0$ for all $t$.
The $m$ paths start from an ordered initial condition $Z^0_1(0)\leq Z^0_2(0)\leq \dotsm\leq Z^0_m(0)$ and perform Brownian motions with drifts $-2b_k$, $k=1,\dotsc,m$, and diffusivity $2$, and interact with each other by one-sided reflections: $Z^0_k(t)$ is reflected to the right off $Z^0_t(k-1)$, $k=1,\dotsc,m$, so that the particles always remain ordered.
In other words, $Z^0_1(t)$ is a Brownian motion with drift $-2b_1$ reflected off the origin and, recursively, $Z^0_k(t)$ is a Brownian motion with drift $-2b_k$ reflected off the lower path $Z^0_{k-1}(t)$, with all the Brownian motions used to run the system being independent.
The system can be constructed explicitly through the Skorokhod reflection mapping, see Section \ref{sec:kpz}.

It is known \cite{Harrison-Williams87} that the system $(Z^0_k)_{k=1,\dotsc,m}$ has a unique stationary distribution $\pi^{(\vec b)}$.
When $m=1$, in which case the system reduces to a Brownian motion reflected off the origin, the stationary distribution is well known to be an exponential with parameter $-2b_1$.
In the general case \cite{FitzGerald-Warren20} showed that, remarkably, $\pi^{(\vec b)}$ can be written in terms of point-to-line (exponential) LPP with boundary sources, see Section \ref{sec:cts}.
When all drifts are different there is also a characterization for the density as a sum of exponentials \cite{Dieker-Moriarty09}.

Let then $\big((\Zeq_k(t))_{t\in\rr}\big)_{k=1,\dotsc,m}$ be a (double-sided) stationary version of our RBMs, having $\pi^{(\vec b)}$ as its fixed time marginals.

\begin{thm}\label{thm:KPZfp}
Consider the Airy process with $m$ wanderers $\aip^{(\infty,\vec{b})}_{m}$ with $0<b_1\leq\dotsm\leq b_m$ and let $\cz^{(\vec b)}_\equ$ be the top path $\Zeq_m$ of the (double-sided) stationary version of the system of reflected Brownian motions with drift with a wall at the origin introduced above.
Then
\begin{equation}
F^{(\infty,\vec b)}_m(r)=\pp\big(\fh(1,0;\cz^{(\vec b)}_\equ)\leq r\big).\label{eq:KPZfp}
\end{equation}
\end{thm}

Since $\cz^{(\vec b)}_\equ$ is stationary and the KPZ fixed point is invariant under spatial shifts, $\fh(1,x;\cz^{(\vec b)}_\equ)$ is stationary in $x$, so \eqref{eq:KPZfp} also holds if we replace $\fh(1,x;\cz^{(\vec b)}_\equ)$ on the right hand side.
Note that if we take all $b_k$'s to infinity then each path $\Zeq_k$ in the stationary system of RBMs, and in particular $\cz^{(\vec b)}_\equ$, converges to $0$ (see Section \ref{sec:limittrans}).
Therefore $\cz^{(\vec b)}_\equ$ can be thought of as a random, $m$-parameter deformation of the flat initial data $\fh(0,\cdot)\equiv0$.

The proof of Theorem \ref{thm:KPZfp} is contained in Section \ref{sec:kpz}.
As a by-product of the proof we get in particular that $\aip^{(\infty,\vec b)}_m$ is itself the KPZ fixed point at time $1$ with initial data obtained from a similar one-sided system of RBMs $(Z^{\uptext{nw}}_m(t))_{t\geq0}$ where the lower one is free (instead of reflected off the origin): for $\cz^{(\vec b)}_{\uptext{nw}}(t)=Z^{\uptext{nw}}_m(-t)$ for $t\leq0$ and $-\infty$ for $t>0$, one has
\begin{equation}\label{eq:airywkpzfp}
\aip^{(\infty,\vec b)}_m(x)-x^2\stackrel{\uptext{(d)}}{=}\fh(1,x;\cz^{(\vec b)}_{\uptext{nw}})
\end{equation}
as processes in $x$.
Versions of this result are known in the context of TASEP and LPP, where $\cz^{(\vec b)}_{\uptext{nw}}$ is usually expressed equivalently as an initial condition constructed from Brownian LPP (see e.g. \cite{Borodin-Ferrari-Sasamoto09,Borodin-Peche08,Corwin-Ferrari-Peche10,Corwin-Liu-Wang16,Imamura-Sasamoto07}, and also Section \ref{sec:cts} below).
The identity 
\begin{equation}
\fh(1,0;\cz^{(\vec b)}_\equ)\stackrel{\uptext{(d)}}{=}\sup_{x\in\rr}\fh(1,x;\cz^{(\vec b)}_{\uptext{nw}})\label{eq:var}
\end{equation}
which follows from \eqref{eq:KPZfp} and \eqref{eq:airywkpzfp} is an instance of the general variational formula satisfied by the KPZ fixed point \cite[Thm. 4.18]{Matetski-Quastel-Remenik17}. 

\subsection{Relation to the Bloemendal--Virag and Korteweg--de Vries PDEs}

The {\it spiked Tracy--Widom distributions} describe the distribution of the largest eigenvalue in self-adjoint random matrix models with a deterministic finite-rank perturbation chosen in such a way that the largest eigenvalue (or several eigenvalues) begin to separate from the bulk.
They have attracted a lot of interest, in part because they characterize the asymptotic fluctuations of the top eigenvalues of Gaussian sample covariance matrices, i.e. the {\it Wishart ensembles}.
For the case of complex Hermitian matrices a Fredholm determinant formula for the spiked Tracy--Widom distribution was obtained initially by Baik, Ben Arous, and Pech\'{e}; the distribution is therefore known in the literature as the BBP distribution.
The Airy process with wanderers $\aip^{(\infty,\vec b)}_m$ provides a dynamic version of the BBP distribution: its one-point marginal distributions coincide exactly with the BBP distribution with spike parameters given by $\vec b$ shifted according to the time where one is focusing.

The most complete description of the spiked TW distributions is due to Bloemendal and Virag \cite{Bloemendal-Virag11, Bloemendal-Virag16}, who showed (among other descriptions) that they are characterized as solutions to a certain boundary value problem for a linear PDE.
Their results hold for a much more general family of spiked TW distributions which is indexed by a real parameter $\beta>0$, where $\beta=1,2,4$ correspond to real, complex, and quaternionic random matrix ensembles, respectively, and other values of $\beta$ may be understood in terms of tridiagonal matrix ensembles as described by Dumitriu and Edelman \cite{Dumitriu-Edelman02}.
As in \cite{Bloemendal-Virag16}, we denote the distribution function for the rank-$m$ spiked TW distribution as $F_{\beta}(x; w_1, w_2, \dots, w_m)$, where $w_1, \dots w_m$ are parameters which describe the strength of the perturbation. When any of the $w_i\to-\infty$, the largest eigenvalue becomes separated from the bulk and its fluctuations take place on a larger scale, thus $F_{\beta}(x; w_1, w_2, \dots, w_m)\longrightarrow 0.$ When $w_1, w_2, \dots w_m\to+\infty$, the perturbation becomes negligible and $F_{\beta}(x; w_1, w_2, \dots, w_m)$ approaches the usual (un-spiked) Tracy--Widom distribution $F_\beta$. The Bloemendal--Virag PDE is
\eq\label{eq:BVpde}
m\frac{\p F}{\p x} +\sum_{j=1}^m\left( \frac{2}{\beta}\frac{\p^2 F}{\p w_j^2} + (x-w_j^2) \frac{\p F}{\p w_j} \right)+\sum_{1\le j<k\le m} \frac{2}{w_k-w_j}\left(\frac{\p F}{\p w_k} - \frac{\p F}{\p w_j}\right)= 0, 
\eeq 
$(x;w_1, \dots, w_m) \in \rr^{m+1}$, and they showed that $F_{\beta}(x; w_1, w_2, \dots, w_m)$ is the unique bounded solution of \eqref{eq:BVpde} which satisfies the boundary conditions
\begin{align}
& F \longrightarrow 1\qquad  {\rm as }  \ x\to+\infty \  {\rm with} \ w_1, w_2, \dots, w_m \  {\rm bounded  \ below}, \label{eq:bc1} \\
&F \longrightarrow 0  \qquad {\rm as \ any} \ w_i\to -\infty \ {\rm with }  \ x  \ {\rm bounded \ above}. \label{eq:bc2}
\end{align}

For the case $\beta = 2$, Baik \cite{Baik06} obtained the following formula, very similar to \eqref{eq:PII1}, for the spiked TW distribution in terms of the Hastings--McLeod solution to PII and the functions $f$ and $g$ defined in \eqref{eq:Lax1}--\eqref{eq:IC}:
\eq\label{eq:spikedb2}
F_{2}(x; w_1, \dots w_m) =\frac{F_2(x)}{\prod_{1\le j<k \le m}(w_k-w_j)} \det\left[(w_j+D_x)^{k-1}f(x,w_j)\right]_{j,k=1}^m.  \\
\eeq
For $\beta=4$ a similar formula is available in the case $m=1$ \cite{Bloemendal-Virag11}:
\eq\label{eq:spikedb4}
F_{4}(2^{-2/3}x; 2^{-1/3}w) = \tfrac{1}{2} \Big((f(x,w)+ g(x,w))E(x)^{-1}+(f(x,w)- g(x,w))E(x)\Big)\!F(x);
\eeq
in the case $w=0$ this was first proved by Wang \cite{Wang08}, who also obtained a Fredholm determinant formula.
For $\beta=1$ there is a rather complicated formula related to PII for the case $m=1$ due to Mo \cite{Mo12}, while no Fredholm determinant formula is known.

In light of \eqref{eq:F1_identity} and given that $F^{(\infty,\vec{b})}_m$ represents a deformation of the Tracy--Widom GOE distribution, and it arises via the Airy process with wanderers, whose marginals are given by the BBP distribution, one might hope that $F^{(\infty,\vec{b})}_m$ is related to the spiked TW distribution with $\beta=1$. This hope is perhaps reinforced by the striking formal similarity between \eqref{eq:PII1} and the formula \eqref{eq:spikedb2} for the BBP distribution, as well as by the fact that the supremum of finitely many non-intersecting Brownian bridges without outliers coincides with largest singular value of a real Wishart matrix \cite{Nguyen-Remenik17}.
We were, however, unable to find such a connection (note in particular that $F^{(\infty,\vec b)}_m$ goes to $0$ as the $b_i$'s go to $0$, while $F_1(x;w_1,\dotsc,w_m)$ does so as the $w_i$'s go to $-\infty$).
Nevertheless, for the case $m=1$, one immediately sees that the right hand side of \eqref{eq:PII1} is exactly twice the second term on the right hand side of \eqref{eq:spikedb4} (recall that $F_1(x)=F(x) E(x)$ in \eqref{eq:PIIfunctions}). The PDE \eqref{eq:BVpde} is linear and is in fact satisfied by each term of \eqref{eq:spikedb4}, so the following corollary is immediate.
It can be checked using formula \eqref{eq:PII1} along with \eqref{eq:Lax1} and \eqref{eq:Lax2}, together with the general identity
\eq\label{eq:PIIId}
v+q^4-(q')^2+xq^2 = 0,
\eeq
which follows from the PII equation satisfied by $q(x)$ (recall $v(x)$ defined in \eqref{eq:PIIfunctions}); uniqueness follows from the same arguments in \cite{Bloemendal-Virag11}\footnote{In \cite{Bloemendal-Virag11} the argument for uniqueness is based on representing the solution of \eqref{eq:BVpde} as the probability that a certain diffusion explodes to $-\infty$; in our case due to the change in the boundary condition this becomes the probability that the diffusion hits the origin, but the rest of the argument remains valid.}.

\begin{cor}[Corollary to the PII formula for $F^{(\infty,b)}_1$]\label{cor:BV}
The distribution function $F^{(\infty,b)}_1(r)$ is the unique bounded solution of the Bloemendal--Virag PDE \eqref{eq:BVpde} with $m=1$, $x=r$, $w=b>0$, and $\beta=4$, in the domain $r\in \rr$, $b>0$, subject to the boundary conditions
\begin{align}
& F \longrightarrow 1\qquad  {\rm as }  \ r\to+\infty \  {\rm with} \ b \ {\rm bounded \ below}, \label{eq:bc1b} \\
&F \longrightarrow 0  \qquad {\rm as } \ b\to0^+ \ {\rm with }  \ r  \ {\rm bounded \ above}. \label{eq:bc2b}
\end{align}
\end{cor}

The appearance of the PDE  \eqref{eq:BVpde} with $\beta=4$ in this context is quite unexpected from a physical point of view, and one might hope that it indicates some relation between the $\beta=1$ and the $\beta=4$ versions of \eqref{eq:BVpde}. We were, again, unable to find such a connection.

For $m>1$, the distribution function $F^{(\infty,\vec{b})}_m(r)$ appears to satisfy a PDE which is nearly identical to \eqref{eq:BVpde}, but has a different coefficient in front of the interaction term.

\begin{conj}\label{conj:BV}
The distribution function $F^{(\infty,\vec{b})}_m(r)$ is the unique bounded solution of the PDE
\eq\label{eq:BVpdealt}
m\frac{\p F}{\p r} +\sum_{j=1}^m\left( \frac{2}{\beta}\frac{\p^2 F}{\p b_j^2} + (r-b_j^2) \frac{\p F}{\p b_j} \right)+\sum_{1\le j<k\le m} \frac{1}{b_k-b_j}\left(\frac{\p F}{\p b_k} - \frac{\p F}{\p b_j}\right)= 0
\eeq 
with $\beta=4$, in the domain $r\in \rr$ and $b_j>0$ for each $j$, subject to the boundary conditions
\begin{align}
& F \longrightarrow 1\qquad  {\rm as }  \ r\to+\infty \  {\rm with} \ b_1, b_2, \dots, b_m \  {\rm bounded  \ below}, \label{eq:bc1c} \\
&F \longrightarrow 0  \qquad {\rm as \ any} \ b_i\to 0^+ \ {\rm with }  \ r  \ {\rm bounded \ above}. \label{eq:bc2c}
\end{align}
\end{conj}

Using \eqref{eq:PII2} together with the equations \eqref{eq:Lax1} and \eqref{eq:Lax2} for $f$ and $g$, it is straightforward to check Conjecture \ref{conj:BV} using a computer algebra system for small values of $m$. We did so for $m=2,3,4$, so in that sense this conjecture can be considered a theorem for $m<5$. 
A proof for general $m$ is surely possible based on these indentities, but it appears not to be straightforward, and is left for future work.

A connection to a different (nonlinear) PDE is arises from the Fredholm determinant formula for $F^{(\infty,\vec{b})}_1(r)$.

\begin{cor}[Corollary to the determinantal formula for $F^{(\vec a,\vec b)}_m$] \label{cor:KdV}
Introduce a rescaling parameter $t>0$, and for fixed $\vec a,\vec b>0$ consider the second logarithmic derivative of rescaled distribution function $F^{(t^{1/3}\vec a,t^{1/3}\vec b)}_m(t^{-1/3} r)$,
\eq
\phi(t,r) \coloneqq \frac{\p^2}{\p r^2}  \log F^{(t^{1/3}\vec a,t^{1/3}\vec b)}_m(t^{-1/3} r).
\eeq
Then $\phi(t,r)$ satisfies the Korteweg--de Vries (KdV) equation
\eq\label{eq:KdV}
\frac{\p \phi}{\p t} + \phi \frac{\p \phi}{\p r} + \frac{1}{12}  \frac{\p^3 \phi}{\p r^3}=0.
\eeq
\end{cor}

This corollary follows directly from the results of \cite{Quastel-Remenik19a} and the Fredholm determinant formula \eqref{eq:Fab-det}. A short proof is presented at the end of Section \ref{sec:DetFormula}.
The question of uniqueness for this equation remains open, see \cite{Quastel-Remenik19a}.

The result is related to the KPZ fixed point characterization of $F^{(\infty,b)}_1(r)$, whose one-point distributions were shown in \cite{Quastel-Remenik19a} to satisfy an integrable PDE which is a two-dimensional extension of KdV: for fixed, deterministic initial data $\fh_0\in\UC$, one has that $\phi(t,x,r)=\p_r^2\log\pp(\fh(t,x;\fh_0)\leq r)$ solves the Kadomtsev-Petviashvili (KP) equation 
\begin{equation}
\frac{\p}{\p r}\!\left(\frac{\p\phi}{\p t}+\frac12\frac{\p\phi^2}{\p r}+\frac1{12}\frac{\p^3\phi}{\p r^3}\right)+\frac14\frac{\p^2\phi}{\p x^2}=0.\label{eq:KP}	
\end{equation}
In fact, by the 1:2:3 scaling invariance of the KPZ fixed point (or rather as a direct consequence of our method of proof), \eqref{eq:KdV} together with Theorem \ref{thm:KPZfp} imply directly the following: 

\begin{cor}\label{cor:kdv-fp}
Let 
\[\phi^{(\vec b)}(t,r)=\p_r^2\log\pp(\fh(t,0;\cz^{(\vec b)}_\equ)\leq r),\]
where $\cz^{(\vec b)}_\equ$ is the top path in the system of RBMs introduced in Theorem \ref{thm:KPZfp}. 
Then $\phi^{(\vec b)}$ solves the KdV equation \eqref{eq:KdV}.
\end{cor}

We stress that in general \eqref{eq:KP} does not hold for random initial data; it is expected to hold only in very special initial conditions such as half-Brownian (see \cite[Ex. 2.7]{Quastel-Remenik19a}).
If $\phi(t,x,r)$ does not depend on $x$ then \eqref{eq:KP} reduces to the KdV equation \eqref{eq:KdV}, but for deterministic initial data, $\fh(t,x,r)$ can be stationary in space only if $\fh_0$ is constant.
Corollary \ref{cor:kdv-fp} provides the first known example beyond flat initial data for which the distribution of the KPZ fixed point is connected to KdV.

Corollaries \ref{cor:BV} and \ref{cor:KdV} together provide a link between the Bloemendal--Virag PDE and the KdV equation, although it does not appear to be straightforward to turn one equation directly into the other.
The KdV equation is one of the most important nonlinear equations in integrable PDE, and one could hope that this connection could help shed light on any further integrable structure of the Bloemendal--Virag PDE.
We note that an analogous statement holds for the KP equation and the BBP distribution with $\beta=2$ (see \cite[Ex. 2.4]{Quastel-Remenik19a}).

\subsection{Connection to LPP with boundary sources}\label{sec:cts}

In this section we briefly discuss connections with a model of last passage percolation with boundary sources introduced in \cite{Borodin-Peche08}.
Let $\{\hat\alpha_j\}_{j\in \nn}$ and $\{\hat\beta_j\}_{j\in \nn}$ be two sequences of real numbers satisfying $\hat\alpha_i+\hat\beta_j>0$ for all $i,j\in \nn$, and consider a family $\{w_{i,j}\}_{i,j\in\nn}$ of independent exponential random variables with means $\ee(w_{i,j})=1/(\hat\alpha_i+\hat\beta_j)$.
The \emph{point-to-point last passage times} are defined as
\eq
L((i,j)\to(i',j'))=\max_{\pi\in\Pi:(i,j)\to(i',j')}{\textstyle\sum_\ell}\ts w_{\pi_\ell}
\eeq
for $i\leq i'$, $j\leq j'$, where the maximum is taken over the set of all directed paths (i.e. taking up and right steps) connecting $(i,j)$ to $(i',j')$.

In \cite{Borodin-Peche08} the authors consider the special case that there exists a fixed $m\in \nn$ such that $\hat\alpha_i=0$ and  $\hat\beta_i=1$ for $i>m$.
 In this case they prove that the process of last passage times from $(1,1)$ to points on an horizontal line $\{(k,N)\}_{k\geq1}$ converges as $N\to\infty$, under the right scaling, to the Airy process with wanderers.
In our setting it is more convenient to look at the last passage times to points on the anti-diagonal line $\{(k,\ell)\!:k,\ell\geq1,\,k+\ell=N\}$, but one expects that this leads to the same limit process.
This is actually known for several related models; in particular it is known for a version of LPP without boundary sources (i.e. $m=0$, where the limit is the Airy$_2$ process; see also \cite{Imamura-Sasamoto04,Corwin-Ferrari-Peche10}) and should follow from the slow decorrelation arguments of \cite{Ferrari08}.
We will only state roughly the expected result: for fixed $\kappa\in(0,1)$ there are constants $\gamma,c_1,c_2,c_3$ such that if $\hat\alpha_i=\kappa+\gamma N^{-1/3}a_i$ and $\hat\beta_i=-\kappa+\gamma N^{-1/3}b_i$ for $i\leq m$, with $a_i+b_i'>0$ for all $i,i'\geq1$, then
\begin{equation}\label{eq:ptpt}
c_1N^{-1/3}\big(L((1,1)\to(\tfrac12N+c_3N^{2/3}u,\tfrac12N-c_3N^{2/3}u))-c_2N\big)\xrightarrow[N\to\infty]{}\aip^{(\vec a,\vec b)}_m(u)-u^2
\end{equation}
in the sense of finite dimensional distributions (here, and below, we implicitly take integer part in non-integer positions when necessary).

Now consider \emph{point-to-line last passage times}, defined for $i,j\geq1$ with $i+j\leq N$ as
\[L^\uptext{line}_N(i,j)=\max_{k,\ell\geq1:\,k+\ell=N}L((i,j)\to(k,\ell)).\]
The validity of \eqref{eq:ptpt} would suggest that $c_1N^{-1/3}\big(L^\uptext{line}_N(1,1)-c_2N\big)$ converges in distribution as $N\to\infty$ to the random variable $\sup_{u\in\rr}\big(\aip^{(\vec a,\vec b)}_m(u)-u^2\big)$ or, in other words, that
\begin{equation}\label{eq:ptln}
\lim_{N\to\infty}\pp\big(c_1N^{-1/3}\big(L^\uptext{line}_N(1,1)-c_2N\big)\leq r)=F^{(\vec a,\vec b)}_m(r).
\end{equation}
This can be regarded as another version of the variational formula \eqref{eq:var}.

A slightly different version of the above last passage times has been related to RBMs in the recent paper \cite{FitzGerald-Warren20}. For a finite $N\in \nn$, they consider the case $\hat\beta_j=\hat\alpha_{N+1-j}$ and prove the following: if $(\zeta^\uptext{eq}_k)_{k=1,\dotsc,N}$ is a vector chosen from the stationary distribution $\pi^{(\vec b)}$ of the system of RBMs $(Z^0_k)_{k=1,\dotsc,N}$ introduced in Section \ref{sec:KPZfluct}, then taking $\hat\alpha_i=\hat\beta_{N+1-i}=b_i$, $i=1,\dotsc,N$, one has
\begin{equation}\label{eq:Yeq}
(\zeta^\uptext{eq}_1,\dotsc,\zeta^\uptext{eq}_N)\stackrel{\text{dist}}{=}(L^\uptext{line}_{N+1}(1,N),\dotsc,L^\uptext{line}_{N+1}(1,1)).
\end{equation}
The distribution of $L^\uptext{line}_{N+1}(1,1)$ is related in \cite{FitzGerald-Warren20} to that of the singular values of a certain symmetric complex Gaussian random matrix.

The parameters used in \cite{FitzGerald-Warren20} are not quite compatible with those used in \cite{Borodin-Peche08}, but they can be chosen so that the difference between the two models is small. Thus it is reasonable to expect that, with a suitable scaling and choice of parameters, the distribution of $L^\uptext{line}_{N+1}(1,1)$ should converge to $F^{(\infty,\vec b)}_m$.
Along these lines, we mention that in \cite{FitzGerald-Warren20} it is also shown that  the distribution of $L^\uptext{line}_{N+1}(1,1)$ (i.e. that of the top RBM in equilibrium) equals that of $\sup_{t\geq0}\lambda_\uptext{max}\big(H(t)-tD\big)$, where $H(t)$ is an $N\times N$ Hermitian Brownian motion, $D$ is a diagonal matrix with entries $D_{ii}=b_i$ and $\lambda_\uptext{max}(A)$ denotes the largest eigenvalue of $A$.
The evolution of the eigenvalues of $H(t)-tD$ can be mapped, under the space-time transformation $(t,\lambda)\longmapsto((1+t)/(1-t),(1-t)\lambda/2)$, to the system of non-intersecting Brownian bridges from Section \ref{sec:nonint} with $\alpha_i=0$ and $\beta_i=b_i$ (see \cite{Aptekarev-Bleher-Kuijlaars05}), but while in the case when all $b_i$'s equal $0$ this supremum has the same distribution as the maximal height of non-intersecting Brownian motions without outliers (see \cite[Sec. 2]{FitzGerald-Warren20}), one can check that this does not hold in general.
The relation between our results and those of \cite{FitzGerald-Warren20} is intriguing, and is left for future work.

\section{Fredholm determinant formula}\label{sec:DetFormula}

In this section we prove Theorem \ref{thm:DetFormula}. We follow the approach introduced in \cite{Corwin-Quastel-Remenik13}, which is based on turning the extended kernel formula \eqref{eq:awextdeter} for the finite dimensional distributions of $\aip^{(\vec a,\vec b)}_m$ into the Fredholm determinant of a certain ``path-integral'' kernel computed in $L^2(\rr)$.
This idea was used in \cite{Corwin-Quastel-Remenik13} in the setting of the Airy$_2$ process, and later generalized to a large class of kernels in \cite{Borodin-Corwin-Remenik15}.

Let $\fH$ denote the \emph{Airy Hamiltonian} $\fH=-\Delta+x$ and consider its associated semigroup $e^{-t\fH}$, which is well defined as an integral operator acting on $L^2(\rr)$ for every $t\geq0$, with integral kernel given by (see \cite{Corwin-Quastel-Remenik13})
\begin{equation}\label{eq:etH}
e^{-t\fH}(x,y)=\frac{1}{\sqrt{4\pi t}}e^{-(y-x)^2/(4t)-t(x+y)/2+t^3/12}.
\end{equation}
Define also two integral operators, $\fA^{(\vec{a},\vec{b})}_{m,s}$ and $\fB^{(\vec{a},\vec{b})}_{m,t}$, $s,t\in\rr$, acting on $L^2(\rr)$ via the kernels 
\begin{align}
\fA^{(\vec{a},\vec{b})}_{m,t}(x,y)&=\frac{1}{2\pi\I}\int_{\Gamma_{-\vec{a}-t\>}}dv\ts e^{-v^3/3+(x+y)v}\prod_{k=1}^m\frac{v-b_k+t}{v+a_k+t}\label{eq:Akernel},\\
\fB^{(\vec{a},\vec{b})}_{m,s}(x,y)&=-\frac{1}{2\pi\I}\int_{\Gamma_{\<\vec{b}-s}}du\ts e^{u^3/3-(x+y)u}\prod_{k=1}^m\frac{u+a_k+s}{u-b_k+s}\label{eq:Bkernel}
\end{align}
(with the contours as in \eqref{eq:awextkernel}).
Then choosing the contours so that $\Re(u+s-v-t)>0$ for $u\in\Gamma_{\<\vec{b}-s}$ and $v\in\Gamma_{\vec{a}-t\>}$, the kernel in \eqref{eq:awextkernel} can be rewritten as
\begin{equation}\label{eq:intAB}
\fK^{(\vec a,\vec b)}_m(s,x;t,y)=-e^{-(t-s)\fH}(x,y)\uno{s<t}+\int_0^{\infty}d\lambda\ts e^{-(s-t)\lambda}\fB^{(\vec{a},\vec{b})}_{m,s}(x,\lambda)\fA^{(\vec{a},\vec{b})}_{m,t}(y,\lambda).
\end{equation}
A straightforward computation using \eqref{eq:etH}--\eqref{eq:Bkernel} gives, for $t\geq0$ and $s\in\rr$,
\begin{equation}
\fA^{(\vec{a},\vec{b})}_{m,s}e^{-t\fH}(x,y)=e^{tx}\fA^{(\vec{a},\vec{b})}_{m,s+t}(x,y),\qquad e^{-t\fH}\fB^{(\vec{a},\vec{b})}_{m,s}(x,y)=e^{ty}\fB^{(\vec{a},\vec{b})}_{m,s-t}(x,y),\label{eq:etHAB}
\end{equation}
and as a consequence we may define $\fA^{(\tilde{a},\tilde{b})}_{m,s}e^{t\fH}(x,y)=e^{-tx}\fA^{(\vec{a},\vec{b})}_{m,s-t}(x,y)$ and $e^{t\fH}\fB^{(\tilde{a},\tilde{b})}_{m,s}(x,y)=e^{-ty}\fB^{(\vec{a},\vec{b})}_{m,s+t}(x,y)$ for all $t\in\rr$, and moreover the property $e^{t_1\fH}e^{t_2\fH}\fB^{(\tilde{a},\tilde{b})}_{m,s}=e^{(t_1+t_2)\fH}\fB^{(\tilde{a},\tilde{b})}_{m,s}$ and $\fA^{(\tilde{a},\tilde{b})}_{m,s}e^{t_1\fH}e^{t_2\fH}=\fA^{(\tilde{a},\tilde{b})}_{m,s}e^{(t_1+t_2)\fH}$ is satisfied (see e.g. \cite[Prop. 1.2]{Quastel-Remenik13a}).
Based on this it is straightforward to check that the hypotheses of \cite[Thm. 3.3]{Borodin-Corwin-Remenik15} are satisfied by the extended kernel $\fK^{(\vec a,\vec b)}_m$, taking (in the notation of that paper) $X=\rr$, $\mu$ the Lebesgue measure, ${\sf W}_{t_i,t_j}(x,y)=e^{-(t_j-t_i){\fH}}(x,y)$, and ${\sf K}_{t_i}(x,y)=\fK^{(\vec{a},\vec{b})}_{m}(t_i,x;t_i,y)$; the key fact is that
\begin{equation}
\fK^{(\vec{a},\vec{b})}_{t_i}(x,y)\coloneqq\fK^{(\vec{a},\vec{b})}_{m}(t_i,x;t_i,y)
=e^{t_i\fH}\fB^{(\vec{a},\vec{b})}_{m,0}\P_0\fA^{(\vec{a},\vec{b})}_{m,0}e^{-t_i\fH}(x,y),\label{eq:KAB}
\end{equation}
which follows from \eqref{eq:intAB} and the above definitions. 
From this and \eqref{eq:awextkernel} we conclude that
\begin{equation}\label{eq:awpathint}
\pp\!\left(\aip^{(\vec{a},\vec{b})}_{m}(t_j)\leq r_j,\,j=1,\dotsc,n\right)=\det\!\left(\fI-\fK^{(\vec{a},\vec{b})}_{t_1}+\bar\P_{r_1}e^{(t_1-t_2)\fH}\bar\P_{r_2}\dotsc\bar\P_{r_n}e^{(t_n-t_1)\fH}\fK^{(\vec{a},\vec{b})}_{t_1}\right)_{L^2(\rr)}.
\end{equation}

Consider next a function $g\in H^1([\ell_1,\ell_2])$ and take $t_1,\dotsc,t_n$ to be a mesh of $[\ell_1,\ell_2]$ and $r_j=g(t_j)$.
The limit of the right hand side of \eqref{eq:awpathint} as the mesh size goes to $0$ can be obtained from the same argument as in the case of Airy$_2$ process, see \cite[Thm. 2]{Corwin-Quastel-Remenik13}, and leads to
\begin{equation}\label{eq:becont}
\pp\!\left(\aip^{(\vec{a},\vec{b})}_{m}\leq g(t)~\forall\,t\in[\ell_1,\ell_2]\right)=\det\!\left(\fI-\fK^{(\vec{a},\vec{b})}_{\ell_1}+\Theta_{[\ell_1,\ell_2]}^{g} e^{(\ell_2-\ell_1)\fH}\fK^{(\vec{a},\vec{b})}_{\ell_1}\right)_{L^2(\rr)},
\end{equation}
where the operator $\Theta_{[\ell_1,\ell_2]}^{g}$ is defined as follows: for $f\in L^2(\rr)$, $\Theta_{[\ell_1,\ell_2]}^{g} f(x)=u(\ell_2,x)$, where $u(\ell_2,\cdot)$ is the solution at time $\ell_2$ of the boundary value problem
\begin{equation}\label{eq:boundpde}
\begin{aligned}
\partial_t u+\fH u&=0\quad\text{for}\ x<g(t),\ t\in(\ell_1,\ell_2)\\
u(\ell_1,x)&=f(x)\uno{x<g(\ell_1)}\\
u(t,x)&=0\quad \text{for}\ x\geq g(t).
\end{aligned}
\end{equation}
We will use now this formula to prove \eqref{eq:Fab-det}.

\begin{proof}[Proof of Theorem \ref{thm:DetFormula}]
The proof follows closely the argument used for the Airy$_2$ process, so we only sketch it.
Set $-\ell_1=\ell_2=L$ in \eqref{eq:becont}.
By employing the Feynman-Kac and the Cameron-Martin-Girsanov formulas (or, alternatively, changing variables in $\Theta_{[-L,L]}^{g}$) the solution of the PDE \eqref{eq:boundpde} can be expressed in terms of the probability that a Brownian motion hits the curve $g(t)-t^2$, and since we are interested in the case $g(t)=t^2+r$, this probability can be computed easily by the refection principle; the result \cite[Eqn. (1.4)]{Corwin-Quastel-Remenik13} is that
\begin{equation}\label{eq:awcont}
\Theta_{[-L,L]}^{(r)}\coloneqq\Theta_{[-L,L]}^{g(t)=t^2+r}=\bar\P_{r+L^2}e^{-2L\fH}\bar\P_{r+L^2}-\bar\P_{r+L^2}\fR_{[-L,L]}^{(r)}\bar\P_{r+L^2},
\end{equation}
where $\fR_{[-L,L]}^{(r)}(x,y)=\frac{1}{\sqrt{8\pi L}}e^{-(x+y-2r-2L^2)^2/8L-(x+y)L+2L^3/3}$.
Using this and the factorization \eqref{eq:KAB} in \eqref{eq:becont} together with the cyclic property of the Fredholm determinant we get
\begin{multline}\label{eq:cycdeter}
\pp\!\left(\aip^{(\vec{a},\vec{b})}_{m}\leq t^2+r~\forall\,t\in\rr\right)=\lim_{L\to\infty}\det\!\left(\fI-(e^{-2L\fH}-\Theta_{[-L,L]}^{(r)})e^{2L\fH}\fK^{(\vec{a},\vec{b})}_{-L}\right)_{L^2(\rr)}\\
=\lim_{L\to\infty}\det\!\left(\fI-\P_0\fA^{(\vec{a},\vec{b})}_{m,0}e^{L\fH}(e^{-2L\fH}-\Theta_{[-L,L]}^{(r)})e^{L\fH}\fB^{(\vec{a},\vec{b})}_{m,0}\P_0\right)_{L^2(\rr)},
\end{multline}
since $e^{2L\fH}\fK^{(\vec{a},\vec{b})}_{-L}=e^{L\fH}\fB^{(\vec{a},\vec{b})}_{m,0}\P_0 \fA^{(\vec{a},\vec{b})}_{m,0}e^{L\fH}$.
Next we write $e^{-2L\fH}-\Theta_{[-L,L]}^{(r)}=\fR_{[-L,L]}^{(r)}-\Omega_L$ with $\Omega_L=\bigl(\fR_{[-L,L]}^{(r)}-{\bar\P}_{r+L^2}\fR_{[-L,L]}^{(r)}{\bar\P}_{r+L^2}\bigr)-\bigl(e^{-2L\fH}-{\bar\P}_{r+L^2}e^{-2L\fH}{\bar\P}_{r+L^2}\bigr)$.
This last operator is to be regarded as an error term, and we have in fact
\begin{equation}\label{eq:OmegaL}
\P_0\fA^{(\vec{a},\vec{b})}_{m,0}e^{L\fH}\Omega_Le^{L\fH}\fB^{(\vec{a},\vec{b})}_{m,0}\P_0\xrightarrow[L\to\infty]{}0
\end{equation}
in trace norm.
The proof of this is the same as that of \cite[Lem. 1.2]{Corwin-Quastel-Remenik13}; the argument in that paper uses estimates based on a steepest descent analysis of the contour integrals defining $\Omega_L$, and the only difference in our case is that the Airy functions there are replaced here by $\fA^{(\vec a,\vec b)}_{m,0}$ and $\fB^{(\vec a,\vec b)}_{m,0}$, but it can be readily checked that the additional rational factors in the integrands in \eqref{eq:Akernel} and \eqref{eq:Bkernel} do not introduce any difficulty, see also \eqref{eq:airyGestimate} below\footnote{The only possible source of trouble is the possibility that in repeating the arguments in \cite{Corwin-Quastel-Remenik13} one may not be able to move to the steepest descent contours without crossing the poles of the rational parts, but the rescaling of the variables in the Airy functions used in that proof implies that those poles end up close to $0$ for large $L$, which precludes any issues.}.
By continuity of the Fredholm determinant with respect to the trace class topology, it follows then from this and \eqref{eq:cycdeter} that, if the limit $\Lambda^{(\vec a,\vec b)}_m=\lim_{L\to\infty}\fA^{(\vec{a},\vec{b})}_{m,0}e^{L\fH}\fR_{[-L,L]}^{(r)}e^{L\fH}\fB^{(\vec{a},\vec{b})}_{m,0}$ exists in trace class in $L^2([0,\infty))$, then
\begin{equation}
\pp\!\left(\aip^{(\vec{a},\vec{b})}_{m}\leq t^2+r~\forall\,t\in\rr\right)
=\det\!\left(\fI-\P_0\Lambda^{(\vec a,\vec b)}_m\P_0\right)_{L^2(\rr)}.\label{eq:wtB1}
\end{equation}
But in fact one has $\fR_{[-L,L]}^{(r)}=e^{-L\fH}\varrho_re^{-L\fH}$, where $\varrho_r$ denotes the reflection operator $\varrho_rf(x)=f(2r-x)$ (as can be checked directly by integration using \eqref{eq:etH} and the formula for the kernel of $\fR_{[-L,L]}^{(r)}$), so $e^{L\fH}\fR_{[-L,L]}^{(r)}e^{L\fH}$ does not depend on $L$ and we get directly that
\begin{equation}
\Lambda^{(\vec a,\vec b)}_m=\fA^{(\vec{a},\vec{b})}_{m,0}\varrho_r\fB^{(\vec{a},\vec{b})}_{m,0}.\label{eq:wtB2}
\end{equation}
Therefore, to get \eqref{eq:Fab-det} we need to show that the right hand side of \eqref{eq:wtB2} coincides with the kernel $\tilde\fB^{(\vec a,\vec b)}_{2^{2/3}r}$ from \eqref{eq:Bab}.
To prove this we note that at this stage, and just as in the contour integral formula for the Airy function, we may replace the contours $\Gamma_{-\vec a\>}$ and $\Gamma_{\<\vec b}$ in the integrals defining $\fA^{(\vec{a},\vec{b})}_{m,0}$ and $\fB^{(\vec{a},\vec{b})}_{m,0}$ by $\I\rr$ (recall all $a_i$'s and all $b_i$'s are positive).
Then
\begin{align}
\Lambda^{(\vec a,\vec b)}_m(x,y)&=\int_{-\infty}^\infty d\lambda\,\itwopii{2}\iint_{(\I\rr)^2}du\,dv\,\frac{e^{u^3/3-(2r-\lambda+y)u}}{e^{v^3/3-(x+\lambda)v}}\prod_{k=1}^m\frac{v-b_k}{v+a_k}\frac{u+a_k}{u-b_k}\\
&=\frac1{2\pi\I}\int_{\I\rr}du\,e^{2u^3/3-(2r+x+y)u}\prod_{k=1}^m\frac{-u-b_k}{-u+a_k}\frac{u+a_k}{u-b_k},
\end{align}
since the $\lambda$ integral yields simply $\delta_0(u+v)$.
This yields the desired formula.
\end{proof}

\begin{proof}[Proof of Corollary \ref{cor:KdV}]
Our distribution function $F^{(t^{1/3}\vec a,t^{1/3}\vec b)}_m(t^{-1/3}r)$ equals the Fredholm determinant of the kernel
$\frac1{2\pi\I}\int_{\langle}dw\,e^{2w^3/3-(x+y+2t^{-1/3}r)w}\prod_{k=1}^m\frac{t^{1/3}a_k+w}{t^{1/3}a_k-w}\frac{t^{1/3}b_k+w}{t^{1/3}b_k-w}$.
Changing variables $(x,y)\longmapsto(t^{-1/3}x,t^{-1/3}y)$ in the Fredholm determinant and $w\longmapsto t^{1/3}w$ in the integral we get that $F^{(t^{1/3}\vec a,t^{1/3}\vec b)}_m(t^{-1/3}r)$ also equals the Fredholm determinant of
\[\wh\fB(x,y)=\frac1{2\pi\I}\int_{\langle}dw\,e^{2w^3t/3-(x+y+2r)w}\prod_{k=1}^m\frac{a_k+w}{a_k-w}\frac{b_k+w}{b_k-w}.\]
This is a Hankel kernel, and it satisfies the differential relations $\frac{\p}{\p r}\wh\fB(x,y)=(\frac{\p}{\p x}+\frac{\p}{\p y})\wh\fB(x,y)$ and $\frac{\p}{\p t}\wh\fB(x,y)=-\frac13(\frac{\p^3}{\p x^3}+\frac{\p^3}{\p y^3})\wh\fB(x,y)$.
The argument in Section 3 of \cite{Quastel-Remenik19a} (see in particular Remark 3.1 there) implies then that $\phi$ satisfies the KdV equation.
\end{proof}

\section{Painlev\'{e} II formula}\label{sec:PII}

The aim of this section is to prove Theorem \ref{thm:PII}.
The first step is to rewrite the right hand side of \eqref{eq:Fab-det} as the Fredholm determinant of a finite rank perturbation of the kernel
\[\tfB_{2^{2/3}r}(x,y)=\inv{2\pi\I}\int_{\langle}du\,e^{2u^3/3-(x+y)u},\] 
see \eqref{eq:Btildeinfty}.
Define the functions
\begin{equation}
\psi^{b_1\dotsm b_k}_{b_1\dotsm b_m}(x)=\inv{2\pi\I}\int_{\langle}du\,e^{2u^3/3-(x+2r)u}\frac{\prod_{j=1}^k(b_j+u)}{\prod_{j=1}^m(b_j-u)},
\end{equation}
and introduce the notation
\begin{equation}
\fB^{b_1\dotsm b_k}_{b_1\dotsm b_m}(x,y)=\psi^{b_1\dotsm b_k}_{b_1\dotsm b_m}(x+y),
\end{equation}
which we will use to denote both the kernel and the integral operator associated to it (note that we have omitted from the notation the dependence of these functions and kernels on $r$).
Observe that $\fB^1_1=\tfB_{2^{2/3}r}$.
For simplicity we will also write $\psi_{b_1\dotsm b_m}\equiv\psi^1_{b_1\dotsm b_m}$ and $\fB=\fB^1_1$.
Denote also by $\fD$ the differentiation operator $\fD f=f'$.

\begin{prop}\label{prop:finiterank-B}
For $0<b_1\leq\dotsm\leq b_m$ we have
\begin{align}
F^{(\infty,\vec b)}_m(r)&=\det\!\left(\fI-\tfB_{2^{2/3}r}-\sum_{j=1}^m\psi_{b_1\dotsm b_j}\otimes\prod_{i=1}^{j-1}\delta_0(b_i-\fD)\right)_{L^2([0,\infty))}\\
&=F_1(2^{2/3}r)\det\!\left(\uno{j=k}-\left(\prod_{i=1}^j(b_i-\fD)\right)(\fI-\tfB_{2^{2/3}r})^{-1}\psi_{b_1\dotsm b_k}(0)\right)_{j,k=1}^m.\label{eq:finiterank-B}
\end{align}
\end{prop}

\begin{proof}
Throughout the proof all Fredholm determinants are computed in $L^2([0,\infty))$.
Note
\[\fB^{b_1\dotsm b_{k+1}}_{b_1\dotsm b_j}=(b_{k+1}-\fD)\fB^{b_1\dotsm b_{k}}_{b_1\dotsm b_j}, \qquad \fB^{b_1\dotsm b_{k}}_{b_1\dotsm b_{j-1}}=(b_j+\fD)\fB^{b_1\dotsm b_{k}}_{b_1\dotsm b_j}.\]
From this and the cyclic property of the Fredholm determinant we get
\eq
\begin{aligned}
F^{(\infty,\vec b)}_m(r)&=\det\!\left(\fI - \fB^{b_1\dotsm b_m}_{b_1\dotsm b_m}\right) = \det\!\left(\fI -(b_m-\fD) \fB^{b_1\dotsm b_{m-1}}_{b_1\dotsm b_m}\right) 
=\det\!\left(\fI -\fB^{b_1\dotsm b_{m-1}}_{b_1\dotsm b_m}(b_m-\fD )\right)  \\
&=\det\!\left(\fI -\fB^{b_1\dotsm b_{m-1}}_{b_1\dotsm b_{m-1}}+(b_m+\fD)\fB^{b_1\dotsm b_{m-1}}_{b_1\dotsm b_m}-\fB^{b_1\dotsm b_{m-1}}_{b_1\dotsm b_m}(b_m-\fD )\right) \\
&=\det\!\left(\fI -\fB^{b_1\dotsm b_{m-1}}_{b_1\dotsm b_{m-1}}-\fB^{b_1\dotsm b_{m-1}}_{b_1\dotsm b_m}\de_0\otimes\de_0\right),
\end{aligned}
\eeq
where in the last equality we have used integration by parts to get $\fB^{b_1\dotsm b_{m-1}}_{b_1\dotsm b_m}\fD +\fD \fB^{b_1\dotsm b_{m-1}}_{b_1\dotsm b_m}=-\fB^{b_1\dotsm b_{m-1}}_{b_1\dotsm b_m}\de_0\otimes\de_0.$
Thus we have reduced the determinant arising in the case of $m$ wanderers to the one for the case of $m-1$ wanderers plus a rank-1 operator. 
We now repeat to reduce to the determinant of the operator for $m-2$ wanderers plus a rank-2 operator:
\begin{align}
&\det\!\left(\fI -(b_{m-1}-\fD)\fB^{b_1\dotsm b_{m-2}}_{b_1\dotsm b_{m-1}}-(b_{m-1}-\fD)(\fB^{b_1\dotsm b_{m-2}}_{b_1\dotsm b_m}\de_0)\otimes\de_0\right)  \\
&\hspace{0.5in}=\det\!\left(\fI -\fB^{b_1\dotsm b_{m-2}}_{b_1\dotsm b_{m-1}}(b_{m-1}-\fD)-(\fB^{b_1\dotsm b_{m-2}}_{b_1\dotsm b_m}\de_0)\otimes\de_0(b_{m-1}-\fD)\right) \\
&\hspace{0.5in}=\det\!\left(\fI- \fB^{b_1\dotsm b_{m-2}}_{b_1\dotsm b_{m-2}}+(b_{m-1}-\fD)\fB^{b_1\dotsm b_{m-2}}_{b_1\dotsm b_{m-1}}-\fB^{b_1\dotsm b_{m-2}}_{b_1\dotsm b_{m-1}}(b_{m-1}-\fD)\right. \\
&\hspace{3.7in}\left.-(\fB^{b_1\dotsm b_{m-2}}_{b_1\dotsm b_m}\de_0)\otimes\de_0(b_{m-1}-\fD)\right) \\
&\hspace{0.5in}=\det\!\left(\fI- \fB^{b_1\dotsm b_{m-2}}_{b_1\dotsm b_{m-2}} - (\fB^{b_1\dotsm b_{m-2}}_{b_1\dotsm b_{m-1}}\de_0)\otimes\de_0-(\fB^{b_1\dotsm b_{m-2}}_{b_1\dotsm b_m}\de_0)\otimes\de_0(b_{m-1}-\fD)\right),
\end{align}
and repeating this $m$ times we arrive at
\[\textstyle\det\!\left(\fI -\fB -\sum_{j=1}^m \left(\fB_{b_1\dotsm b_j}\de_0\otimes \de_0\right)\prod_{i=1}^{j-1} (b_i-\fD)\right),\]
which gives the first formula.
The second one follows from the first one and the matrix determinant lemma together with \eqref{eq:TWdistr2}, which gives $\det(\fI-\wt\fB_{2^{2/3}r})=F_1(2^{2/3}r)$.
\end{proof}

The next step is to turn \eqref{eq:finiterank-B} into a formula involving only the functions $\psi_{b_i}$.

\begin{prop}\label{prop:finiterank-B-2}
For $0<b_1\leq\dotsm\leq b_m$ we have
\begin{equation}
F^{(\infty,\vec b)}_m(r)=\frac{1}{\prod_{1\leq j<k\leq m}(b_k-b_j)}F_1(2^{2/3}r)\det\!\left(b_k^{j-1}-\fD^{j-1}(\fI - \fB)^{-1}\psi_{b_k}(0)\right)_{j,k=1}^m.
\end{equation}
\end{prop}

\begin{proof}
For a given $\vec a=(a_1,\dotsc,a_m)$ introduce the Vandermonde determinant
\[\textstyle\Delta_m(\vec a)=\det\!\big(a_k^{j-1}\big)_{j,k=1}^m=\prod_{1\leq j<k\leq m}(a_k-a_j).\]
We will use repeatedly the formula
\begin{equation}
\psi_{a_1\dotsm a_\ell}=\tfrac{1}{a_\ell-a_{\ell-1}}\big(\psi_{a_1\dotsm a_{\ell-1}}-\psi_{a_1\dotsm a_{\ell-2}a_\ell}\big),\label{psipsipsi}
\end{equation}
valid for any $\vec a=(a_1,\dotsc,a_\ell)$ with $a_{\ell-1}\neq a_\ell$, to manipulate the determinant on the right hand side of \eqref{eq:finiterank-B} (where for $\ell=2$ the quantity $\psi_{a_1\dotsm a_{\ell-2}a_\ell}$ is interpreted as $\psi_{a_2}$).
Applying the formula with $\vec a=(b_1,\dotsc,b_m)$, the last column of the determinant becomes
\[\left({\bf 1}_{j=m}-\tfrac1{b_m-b_{m-1}}\left({\textstyle\prod_{\ell=1}^{j-1}(b_\ell-\fD)}\right)(\fI - \fB)^{-1}(\psi_{b_1\dotsm b_{m-1}}-\psi_{b_1\dotsm b_{m-2}b_m})(0)\right)_{j=1,\dotsc,m}.\]
Now we take the factor $\frac1{b_m-b_{m-1}}$ outside the determinant and then subtract the $(m-1)$-th column in order to remove the term involving $\psi_{b_1\dotsm b_{m-1}}$.
The result is the same determinant which we started with, multiplied by $\frac1{b_m-b_{m-1}}$, but with the last column replaced by
\[\left((b_m-b_{m-1}){\bf 1}_{j=m}-{\bf 1}_{j=m-1}+\left({\textstyle\prod_{\ell=1}^{j-1}(b_\ell-\fD)}\right)(\fI - \fB)^{-1}\psi_{b_1\dotsm b_{m-2}b_m}(0)\right)_{j=1,\dotsc,m}.\]
Note that we have managed to remove the factor involving $b_{m-1}$ from $\psi_{b_1\dotsm b_{m}}$.
We repeat now the same procedure inductively, using \eqref{psipsipsi} with $\vec a=(b_1,\dotsc,b_{m-\ell},b_m)$, from $\ell=2$ up to $\ell=m-1$.
The result is a determinant which is now premultiplied by $\prod_{\ell=1}^{m-1}\frac1{b_m-b_\ell}$ and where the last column is replaced by
\[\textstyle\left((-1)^{m+j}\prod_{\ell=1}^{j-1}(b_m-b_{\ell})+(-1)^{m}\left({\textstyle\prod_{\ell=1}^{j-1}(b_\ell-\fD)}\right)(\fI - \fB)^{-1}\psi_{b_m}(0)\right)_{j=1,\dotsc,m}.\]
We may repeat now the whole procedure in each of the columns $m-\ell$, from $\ell=2$ to $\ell=m-2$.
As we do this we keep pulling factors outside the determinant, and once everything is done the resulting prefactor is clearly $\prod_{1\leq\ell<\ell'\leq m}\frac1{b_{\ell'}-b_\ell}=\frac1{\Delta_m(\vec b)}$.
The end result is that the determinant on the right hand side of \eqref{eq:finiterank-B} has been replaced by
\begin{multline}\label{deDet2}
\textstyle\frac1{\Delta_m(\vec b)}\det\!\left((-1)^{j+k}U_{j,k}+(-1)^k\left({\textstyle\prod_{\ell=1}^{j-1}(b_\ell-\fD)}\right)(\fI - \fB)^{-1}\psi_{b_k}(0)\right)_{j,k=1}^m\\
=\textstyle\frac1{\Delta_m(\vec b)}\det\!\left(U_{j,k}+(-1)^j\left({\textstyle\prod_{\ell=1}^{j-1}(b_\ell-\fD)}\right)(\fI - \fB)^{-1}\psi_{b_k}(0)\right)_{j,k=1}^m,
\end{multline}
with
\[U_{j,k}=\prod_{\ell=1}^{j-1}(b_k-b_{\ell}){\bf 1}_{j\leq k}.\]

Let us now denote by $e_k^{(n)}$ the $k$-th elementary symmetric polynomial on the $n$ variables $(b_1, \dots, b_n)$, i.e.
\eq\label{def:ek}
e_k^{(n)}=e_k(b_1, \dots, b_n) = \sum_{1\le i_1 < i_2 < i_k \le n} b_{i_1}\dots b_{i_k}.
\eeq
Then we may expand
\[
\prod_{\ell = 1}^{j-1} (b_\ell - \fD) = \sum_{\ell= 0}^{j-1} (-1)^\ell e_{j-1-\ell}^{(j-1)} \fD^{\ell}.
\]
Note that the coefficient of $\fD^{j-1}$ is $(-1)^{j-1} e_0^{(j-1)} = (-1)^{j-1}$.
We then proceed using row operations to eliminate all terms involving the factor $\fD^\ell$ with $\ell<j-1$ from the $j$-th row on the determinant on the right hand side of \eqref{deDet2}.
This is equivalent to multiplication on the left by a lower triangular matrix $L$, defined as the product
\[
L = L_{m-1}L_{m-2} \dotsm L_{1},
\]
where $L_k$ consists of $1$'s along the diagonal and $0$'s elsewhere except in the $k$-th column below the main diagonal; those entries are 
\[
(L_k)_{k+\ell,k} =(-1)^{\ell-1} e_{\ell}^{(k+\ell-1)}, \qquad \ell=1, \dots , m-k.
\]
The right hand side of \eqref{deDet2} then equals
\begin{equation}\label{eq:LUjk}
\tfrac1{\Delta_m(\vec b)}\det\!\left((LU)_{jk}-\fD^{j-1}(\fI - \fB)^{-1}\psi_{b_k}(0)\right)_{j,k=1}^m.
\end{equation}

All that remains is to compute $LU$.
But it turns out that this is just the $LU$ decomposition of the Vandermonde matrix in $\vec b$:
\[(LU)_{jk}=b_k^{j-1}.\]
In fact, it is known \cite{Oruc-Phillips00} that\footnote{The formulas presented here differ from the ones in \cite[Thm. 2.2]{Oruc-Phillips00} by transpose and moving the diagonal part of the decomposition from $\bar L$ to $\bar U$.}
\[b_k^{j-1}=(\bar L\bar U)_{jk}\]
for 
\[\bar U_{jk} = \prod_{\ell=1}^{j-1}(b_k - b_{\ell}), \quad j\le k, \qqand
\bar L_{jk} = \tau_{k-j}(b_1, \dots b_k),\quad j\ge k,\]
where $\tau_n(b_1, \dots b_k)$
\eq\label{def:homog}
\tau_n(b_1, \dots b_k) \coloneqq \sum_{1\le i_1\le i_2 \le \cdots \le i_n \le k} b_{i_1} b_{i_2} \cdots b_{i_n}
\eeq
is the complete homogeneous symmetric polynomial of order $n$ in $k$ variables. 
$\bar U$ is exactly the matrix $U$ which we obtained above.
For the lower triangular part recall that the diagonal entries of $L$ equal $1$ and note that, from its definition, its entries below the diagonal satisfy the recursion
\begin{align}
\textstyle L_{k+\ell,k} &=- \sum_{\ell'=0}^{\ell-1}(-1)^{\ell+\ell'} L_{k+\ell', k} e_{\ell-\ell'}^{(k+\ell-1)} \\
&= \sum_{j=1}^{\ell}(-1)^{j+1} L_{k+\ell-j, k} e_{j}^{(k+\ell-1)}
\end{align}
for $\ell=1, \dotsc , m-k$.
We claim that the complete homogeneous symmetric polynomials \eqref{def:homog} satisfy the same recursion, i.e.
\eq\label{eq:taurecursion}
\textstyle \tau_{\ell}(b_1,\dotsc,b_k) = \sum_{j=1}^{\ell}(-1)^{j+1} \tau_{\ell-j}(b_1,\dotsc,b_k) e_{j}^{(k+\ell-1)},\qquad \ell=1, \dotsc , m-k.
\eeq
When the number of variables for $\tau_j$ and $e_j$ are the same, this is a well known identity relating the elementary and complete symmetric polynomials, see e.g, \cite[Eqn. 2.6']{MacDonald95}, which is easily proven using generating functions. Here the same proof works.
The two families of polynomials have generating functions
$\sum_{j=0}^\infty (-t)^j e_j^{k+\ell-1} = \prod_{i=1}^{k+\ell-1}(1-b_i t)$, $\sum_{j=0}^\infty t^j \tau_j(b_1,\dotsc,b_k) = \prod_{i=1}^{k}(1-b_i t)^{-1}$.
Taking their product, we find 
\eq
\bigg(\sum_{j=0}^\infty (-t)^j e_j^{k+\ell-1} \bigg)\bigg(\sum_{j=0}^\infty t^j \tau_j(b_1,\dotsc,b_k) \bigg) =  \prod_{i=k+1}^{k+\ell-1}(1-b_i t).
\eeq
The $t^\ell$ term on the left hand side of the above equation is exactly the difference between the left and right hand sides of \eqref{eq:taurecursion} and the $t^\ell$ term on the right side is zero, so \eqref{eq:taurecursion} follows.

Since the sequences $(L_{k+\ell,k})_{\ell=0}^{m-k}$ and $(\bar L_{k+\ell,k})_{\ell=0}^{m-k}$ satisfy the same recursion and have the same initial condition $L_{kk} = \bar L_{kk} =1$, it follows that $L=\bar L$.
We deduce then that $(LU)_{jk} = (\bar L \bar U)_{jk}= b_k^{j-1}$ which, along with \eqref{eq:LUjk}, proves the proposition.
\end{proof}

In order to compare the formula given in the previous result with our PII formula for $F^{(\infty,\vec b)}_m$ we will use the following formulas for the functions $f$ and $g$, given in \cite{Baik06}
\begin{equation}
\begin{aligned}
f(2^{2/3}r,2^{1/3}b)&=1-(\fI-\fB^2)^{-1}\fB\psi_b(0),\\
g(2^{2/3}r,2^{1/3}b)&=(\fI-\fB^2)^{-1}\psi_b(0).
\end{aligned}
\end{equation}
We note here that the PII solution $u$ presented in \cite{Baik06} differs from the function $q$ we defined in \eqref{eq:HM} by sign. Comparing the Lax pair \eqref{eq:Lax1} and \eqref{eq:Lax2} with the one presented in \cite{Baik06}, we find that our function $f$ is the same as the function $f$ from \cite{Baik06}, whereas our $g$ differs from the one in \cite{Baik06} by sign. The above formulas then follow from Theorem 1.1 and the remark following Lemma 1.4 in that paper.
Then since $(\fI-\fB^2)^{-1}(\fI+\fB)=(\fI-\fB)^{-1}$ we get
\begin{equation}
\label{eq:fmgformula}
f(2^{2/3}s,2^{1/3}b)-g(2^{2/3}s,2^{1/3}b) = 1-(\fI-\fB)^{-1}\psi_b(0).
\end{equation}

\begin{proof}[Proof of Theorem \ref{thm:PII}]
In view of the formula in Proposition \ref{prop:finiterank-B-2} and \eqref{eq:fmgformula}, we need to prove that
\begin{equation}\label{deDet}
\det\!\left(b_k^{j-1}-\fD^{j-1}(\fI - \fB)^{-1}\psi_{b_k}(0)\right)_{j,k=1}^m=\det\!\begin{pmatrix}\left(b_k + D_r\right)^{j-1}\left(1-(\fI - \fB)^{-1}\psi_{b_k}(0)\right)\end{pmatrix}_{j,k=1}^m.
\end{equation}
The key will be to prove that for any $b$ and any $j=1,2,\dots,m-1$,
\begin{equation}
\alpha_j\big(1-(\fI - \fB)^{-1} \psi_b(0)\big)+b^{j}-\fD^{j}(\fI - \fB)^{-1} \psi_b(0)
=(b+D_r)\left(b^{j-1}-\fD^{j-1}(\fI - \fB)^{-1}\psi_{b}(0)\right),\label{eq:bbsHb}
\end{equation}
where $\alpha_j$ is a constant which does not depend on $b$. 
To see why, denote by $G$ and $H$ the matrices appearing in left and right hand sides of \eqref{deDet}, and note first that their first rows are trivially equal, while \eqref{eq:bbsHb} with $j=1$ states exactly that the second row of $H$ equals the second row of $G$ plus a multiple of the first row of $G$. 
We claim that in fact, for $j=2,\dotsc,m$, the $j$-th row of $H$ equals the $j$-th row of $G$ plus a linear combination of the first $j-1$ rows of $G$, which we can prove by induction. Assume this is true for the $(j-1)$th row, i.e.,
\[H_{j-1,k} = G_{j-1,k}+\sum_{\ell = 1}^{j-2} \beta_\ell G_{\ell k},\]
where the coefficients $\beta_\ell$ do not depend on $k$.
Applying $(b_k+D_r)$ to the left hand side of the above equation transforms $H_{j-1,k}$ into $H_{jk}$. Applying it to the right hand side, we can use \eqref{eq:bbsHb}, which implies that $(b_k+D_r)G_{\ell k}=G_{\ell +1,k}+\alpha_{\ell}G_{1k}$. The result is
\[H_{jk} = G_{jk}+\alpha_{j}G_{1k}+\sum_{\ell = 1}^{j-2}\beta_\ell(G_{\ell+1,k}+\alpha_{\ell+1}G_{1k})
=G_{jk}+(\alpha_2+\dotsm+\alpha_{j})G_{1k}+\sum_{\ell = 2}^{j-1}\beta_{\ell-1}G_{\ell k},\]
which completes the induction step. This implies that $G$ and $H$ differ from one another by elementary row operations, and thus $\det(G) = \det(H)$.

Let us then prove \eqref{eq:bbsHb}. 
The argument will be reminiscent of some of the computations in \cite{Tracy-Widom94}.
The identity is equivalent to 
\begin{multline}
\label{eq:com1}
\alpha_j\big(1-(\fI - \fB)^{-1} \psi_b(0)\big) = \fD^{j-1}\Big(\fD(\fI - \fB)^{-1}\psi_b(0)- b(\fI - \fB)^{-1}\psi_b(0)\\
-\left(D_r(\fI - \fB)^{-1}\right)\psi_b(0)-(\fI - \fB)^{-1}D_r\psi_b(0)\Big).
\end{multline}
Now we use the formulas $D_r(\fI - \fB)^{-1}=(\fI - \fB)^{-1}D_r\fB(\fI - \fB)^{-1}$, $D_r\fB=2\fB'$ (here $\fB'=\fD\fB$) and $D_r\psi_b=2\psi_b'$ to get that the RHS of \eqref{eq:com1} equals
\[
\fD^{j-1}\!\left(\fD(\fI - \fB)^{-1}\psi_b(0)\tsm-\tsm b(\fI - \fB)^{-1}\psi_b(0)\tsm-\tsm2(\fI - \fB)^{-1}\fB'(\fI - \fB)^{-1}\psi_b(0)\tsm-\tsm2(\fI - \fB)^{-1}\psi_b'(0)\right).
\]
Next we use the general formula $[\fD,(\fI - \fB)^{-1}]=(\fI - \fB)^{-1}[\fD,\fB](\fI - \fB)^{-1}$, where $[\cdot,\cdot]$ denotes the commutator, together with integration by parts, which gives $[\fD,\fB]=2\fB'+\fB\delta_0\otimes\delta_0$, to see that the above equals
\[
\fD^{j-1}\left((\fI - \fB)^{-1}\fB(0,0)\,(\fI - \fB)^{-1}\psi_b(0)- b(\fI - \fB)^{-1}\psi_b(0)-(\fI - \fB)^{-1}\psi_b'(0)\right).
\]
But $\psi_b'=\fB\delta_0-b\psi_b$, so the last expression equals
\begin{multline}
\fD^{j-1}\left((\fI - \fB)^{-1}\fB(0,0)\,(\fI - \fB)^{-1}\psi_b(0)- (\fI - \fB)^{-1}\fB(0,0)\right) \\
=-\fD^{j-1}(\fI - \fB)^{-1}\fB(0,0)\left(1-(\fI - \fB)^{-1}\psi_b(0)\right).
\end{multline}
Thus  \eqref{eq:bbsHb} holds with $\alpha_j=-\fD^{j-1}(\fI - \fB)^{-1}\fB(0,0)$.
\end{proof}

\section{KPZ fixed point characterization}\label{sec:kpz}

The aim of this section is to prove Theorem \ref{thm:KPZfp}.
The proof has several steps, and we start by describing them briefly.
Our first task will be to construct an initial condition for TASEP which approximates $\cz^{(\vec b)}_\equ$ and for which exact formulas can be obtained.
However, these exact methods work best for TASEP initial data which is one-sided (meaning that in the particle system there is a rightmost particle), so we will actually first approximate an initial condition $\fh^{(\vec b)}_0$ which corresponds essentially to the top path $(Z_m(t))_{t\geq0}$ of the one-sided system of RBMs, all started at the origin.
The next step will be to compute the asymptotics of the TASEP height function with this choice of initial data, which gives a Fredholm determinant formula for the KPZ fixed point with initial condition given by $\fh^{(\vec b)}_0$.
Finally we will focus on a location far from the origin and take a limit in the formula to recover, on one side, the KPZ fixed point with initial condition given by $\cz^{(\vec b)}_\equ$, and on the other the Fredholm determinant appearing in \eqref{eq:Fab-det}.

\vskip6pt

\begin{rmk}\label{rem:multi}
In order to not overload notation, throughout the section we will only write formulas for the one-point distribution of TASEP and the KPZ fixed point.
However, everything we will do can be extended to the multi-point distributions as in \cite{Matetski-Quastel-Remenik17} without any issues.
\end{rmk}

Throughout this section all Fredholm determinants are computed in either $\ell^2(\zz)$ (for discrete kernels) or $L^2(\rr)$ (for continuous ones).

\subsection{TASEP formulas}

We begin with a brief description of the TASEP formulas derived in \cite{Matetski-Quastel-Remenik17}.
Recall the particle system description of the process, given in Section \ref{sec:KPZfluct}.
We will only be interested in the case where the initial data (and thus the process at all times) has a rightmost particle (i.e. a rightmost occupied site).
The process may be represented by specifying the locations $X_1(t)>X_2(t)>\dotsm$ of the particles at each given time $t$.
The transition probabilities for TASEP with $N$ particles were found in \cite{Schutz97} using the coordinate Bethe ansatz. They are given by 
\begin{equation}\label{eq:Green}
\pp_{X_0}(X_t(1)= x_1,\ldots,X_t(N)=x_N)=\det(F_{i-j}(t,x_{N+1-i}-X_0(N+1-j)))_{1\leq i,j\leq N}
\end{equation}
with
\begin{equation}\label{eq:Fn}
F_{n}(t,x)=\frac{(-1)^n}{2\pi \I} \oint_{\Gamma_{0,1}}\d w\,\frac{(1-w)^{-n}}{w^{x-n+1}}\psi_t(w),
\end{equation}
where $\psi_t(w)=e^{t(w-1)}$ and $\Gamma_{0,1}$ is any positively oriented simple loop which includes $w=0$ and $w=1$.
However, when computing the scaling limit \eqref{eq:TASEPtoFP} one actually needs a usable formula for the distribution function $\pp_{X_0}(X_t(n)>a)$.
In \cite{Borodin-Ferrari-Prahofer-Sasamoto07,Sasamoto05}, \eqref{eq:Green} was turned into a Fredholm determinant formula for this distribution function, which depends on a kernel which can be computed based on solving a certain biorthogonalization problem for a family of functions constructed out of $F_n$ and the specific initial data under consideration.
In those papers the biorthogonalization problem was solved for half-periodic initial data ($X_0(i)=-2i$, $i\geq1$), which ultimately allowed them to prove convergence to the Airy$_1$ process in the case of fully periodic initial data ($X_0(i)=-2i$, $i\in\zz$).
The solution for general (one-sided) initial data was obtained much later in \cite{Matetski-Quastel-Remenik17}, and it involves transition probabilities of a random walk forced to hit a curve defined by the initial data.
After postprocessing, it leads to the following.
Define
\begin{align}
 \SM_{-t,-n}(z_1,z_2) &=\frac{1}{2\pi\I} \oint_{\Gamma_0}\d w\, \frac{(1-w)^{n}}{2^{z_2-z_1} w^{n +1 + z_2 - z_1}}e^{tw/2}\tts\psi_t(w),\label{def:sm}\\
 \SN_{-t,n} (z_1,z_2) &=\frac{1}{2 \pi \I} \oint_{\Gamma_{0}} \d w\,\frac{(1-w)^{z_2-z_1 + n - 1}}{2^{z_1-z_2} w^{n}} e^{-tw/2}\frac1{\psi_t(1-w)},\label{def:sn}
\end{align}
where $\Gamma_0$ is now a positively oriented simple loop which includes the pole at $w=0$ but not the one at $w=1$.
Let also $\tau$ be the hitting time of the strict epigraph of the discrete curve $\big((k,X_0(k+1))\big)_{k=0,\dotsc,n-1}$ by a discrete time random walk $(V_n)_{n\geq0}$ with Geom$[\frac12]$ jumps supported on the strictly negative integers, and define
\begin{equation}\label{eq:sepi23}
{\SN}_{-t,n}^{\epi(X_0)}(z_1,z_2) = \ee_{B_0=z_1}\!\left[ \SN_{-t,n - \tau}(V_{\tau}, z_2){\bf 1}_{\tau<n}\right].
\end{equation}
Then for any $n\geq1$, $t>0$ and $a\in\zz$,
\begin{equation}\label{eq:TASEPformula}
\pp_{X_0}\big(X_t(n)>a\big)=\det\!\left(I-\bar\P_a(\SM_{-t,-n})^*\SN^{\epi(X_0)}_{-t,n}\bar\P_a\right).
\end{equation}

Next we introduce a related particle system: \emph{discrete time PushTASEP with left geometric jumps}.
As before, we have a one-sided system of particles evolving on $\zz$, which we label as  $X^\circ_t(1)>X^\circ_0(2)>\dotsm$.
Each (discrete time) step of the process is run as follows: particles are updated from right to left; each particle makes a Geom$[q]$ jump to the left (supported in $\{0,-1,-2,\dotsc\}$, this means that a jump to the left of size $k$ has probability $(1-q)^kq$), and then pushes all the particles it finds in its way in order to keep the ordering.
The transition probabilities for the $N$-particle system can be expressed just like for TASEP: for $\ell\in\nn$,
\begin{equation}\label{eq:GreenPush}
\pp_{X_0^\circ} (X^\circ_\ell(1)= x_1,\ldots,X^\circ_\ell(N)=x_N)=\det(F^\circ_{i-j}(\ell,x_{N+1-i}-X_0(N+1-j)))_{1\leq i,j\leq N},
\end{equation}
where $F^\circ_n$ is the same function as in \eqref{eq:Fn} except that $\psi_t$ is replaced by $\psi^\circ_\ell(w)=\left(\frac{q}{1-(1-q)\tts w^{-1}}\right)^\ell$ and the contour needs to encircle $1-q$ as well as $0$ and $1$ including the pole at $1-q$, see \cite{Dieker-Warren08}.

The TASEP initial data which we want to consider corresponds to the one obtained by starting with some given, one-sided initial condition $X^\circ_0$ and applying $m$ discrete time PushTASEP steps, where we allow the parameter $q=q_k$ in the $k$-th PushTASEP step to depend on $k$.
We will denote this choice as $X^\circ_m(i)$, the dependence on the parameters $q_k$, $k=1,\dotsc,m$, remaining implicit.
The transition probabilities for TASEP with $N$ particles can be obtained directly by convolving \eqref{eq:Green} and \eqref{eq:GreenPush} ($m$ times with $\ell=1$) and it follows from an argument based on \cite{Dieker-Warren08,Johansson10} that the result is given once again by a formula like those two, where the function $\psi$ appearing in the contour integral formula \eqref{eq:Fn} is now replaced by
\begin{equation}
\psi^\circ_{m,t}(w)=\psi^\circ(w)\psi_t(w)=e^{t(w-1)}\prod_{j=1}^m\tfrac{q}{1-(1-q_j)\tts w^{-1}};\label{eq:psimt}
\end{equation}
the argument is provided in \cite{Matetski-Quastel-Remenik20}.
Moreover, the same paper shows that the whole (long) derivation which goes from \eqref{eq:Green} to \eqref{eq:TASEPformula} is valid when $\psi_t$ is replaced by $\psi^\circ_{m,t}$.
As a consequence, we have:

\begin{prop}\label{prop:TASEPdTASEP}
Consider continuous time TASEP with initial data $X^\circ_m$ as described above.
Then for any $n\geq1$, $t>0$ and $a\in\zz$,
\begin{equation}\label{eq:TASEPformula-o}
\pp_{X^\circ_m}\!\big(X_t(n)>a\big)=\det\!\left(I-\bar\P_a(\SM^\circ_{-t,-n})^*\SN^{\circ,\epi(X^\circ_0)}_{-t,n}\bar\P_a\right),
\end{equation}
where ${\SN}_{-t,n}^{\circ,\epi(X_0)}(z_1,z_2) = \ee_{B_0=z_1}\!\left[ \SN^\circ_{-t,n - \tau}(V_{\tau}, z_2){\bf 1}_{\tau<n}\right]$ with $V_n$ and $\tau$ defined as in \eqref{eq:sepi23} and $\SM^\circ_{-t,-n}$ and $\SN^{\circ,\epi(X^\circ_0)}_{-t,n}$ are defined like \eqref{def:sm} and \eqref{def:sn} with $\psi_t$ replaced by $\psi^\circ_{m,t}$ and the contour $\Gamma_0$ encircling $0$ and $1-q_k$ for each $k=1,\dotsc,m$, but not $1$.
\end{prop}

In our application to the proof of Theorem \ref{thm:KPZfp} we will be interested in the case where the initial configuration from which the $m$ discrete time PushTASEP steps are run is the half-periodic one, i.e. $X^\circ_0(i)=-2i$, $i\geq1$.
However, since it introduces no additional difficulty, throughout the next two sections we work with a general choice of $X^\circ_0$, under the assumption that it converges under diffusive scaling.

\subsection{Limit of the initial data}\label{sec:inilim}

We need to compute the limit of $X^\circ_m$ as initial data under the scaling which takes TASEP to the KPZ fixed point.
From \cite[Eqn. (3.4)]{Matetski-Quastel-Remenik17}, what we need is to compute the limit as $\ep\to0$ of
\begin{equation}
\fX^\ep_m(x)\coloneqq-\ep^{1/2}(X^\circ_m(\ep^{-1}x)+2\ep^{-1}x-1);\label{eq:X0sc}
\end{equation}
the corresponding initial condition for the KPZ fixed point will then be given by $\lim_{\ep\to0}\fX^\ep_m(-x)$ if $x\leq0$ and $-\infty$ otherwise.

The evolution of the discrete time PushTASEP particle system can be constructed recursively as follows: for $k\geq1$, and given the state of the process at time $k-1$, let
\[X^\circ_k(i)=\min\{X^\circ_k(i-1)-1,X^\circ_{k-1}(i)\}-\xi^k_i,\quad i\geq1,\]
where we set with $X^\circ_k(0)=\infty$ for all $k$ (so that $X^\circ_k(1)=X^\circ_{k-1}(1)-\xi^k_1$), and where the $(\xi^k_i)_{i,k\geq1}$ are i.i.d. Geom$[q_k]$ random variables.
Recentering around $-2i+1$ and reflecting, we define $\wh X^\circ_k(i)=-(X^\circ_k(i)+2i-1)$, so that the above recursion becomes
\[\wh X^\circ_k(i)=\max\{\wh X^\circ_k(i-1),\wh X^\circ_{k-1}(i)+1\}+(\xi^k_i-1),\quad i\geq1.\]
Now we fix $k\geq1$ and analyze $(\wh X^\circ_k(i))_{i\geq1}$ as a Markov chain in $i$.
From the above recursion, this Markov chain can be thought of as a (time-inhomogeneous) random walk with Geom$[q_k]$ steps (at time $k$) supported on $\{-1,0,1,\dotsc\}$, except that every time it hits the shifted curve $(\wh X^\circ_{k-1}(i+1))_{i\geq1}$ it receives an extra push up by $1$ on the next time step (note that from the dynamics and the initial ordering of the particles, $\wh X^\circ_k(i-1)$ is always greater than or equal to $\wh X^\circ_{k-1}(i)$).
With this in mind, defining
\[S_k(i)=\sum_{\ell=1}^i(\xi^k_i-1)\qqand A_k(i)=\sup_{\ell=1,\dotsc,i}\big(S_k(\ell-1)-\wh X^\circ_{k-1}(\ell)-1\big)^-\]
(here $x^-=-x\uno{x<0}$), one checks that
\begin{equation}\label{eq:whXSA}
\wh X^\circ_k(i)=S_k(i)+A_k(i),\quad i\geq2.
\end{equation}
Defining $\fX^\ep_k(x)$ as in \eqref{eq:X0sc}, \eqref{eq:whXSA} yields
\[\fX^\ep_k(x)=S^\ep_k(x)+A^\ep_k(x),\]
with
\[S^\ep_k(x)=\ep^{1/2}S_k(\ep^{-1}x)\qqand A^\ep_k(x)=\sup_{\ell=1,\dotsc,\ep^{-1}x}\big(S^\ep_k(x-\ep)-\fX^\ep_{k-1}(x)-1\big)^-\]
for $x\in\ep\nn$.
We extend these functions linearly to all $x\geq0$ and then define recursively $\tilde A^\ep_k(x)=\sup_{y\in[0,x]}\big(S^\ep_k(y)-\fX^\ep_{k-1}(x)-1\big)^-$ and
\begin{equation}\label{eq:tildeXSA}
\tilde\fX^\ep_k(x)=S^\ep_k(x)+\tilde A^\ep_k(x).
\end{equation}
It is easy to see that $|\fX^\ep_k(x)-\tilde\fX^\ep_k(x)|\leq2\ep^{1/2}$, so we may just compute the limit of $\tilde\fX^\ep_k$.

Recall Skorokhod's reflection mapping (see e.g. \cite[Lem. VI.2.1]{Revuz-Yor99}) which, given two continuous functions $f,z\!:[0,\infty)\longrightarrow\rr$, $f(0)\geq z(0)$, constructs the reflection of $f$ off $z$ as
\begin{equation}\label{eq:sk-refl}
{\cal R}_zf(t)=f(t)+\sup_{s\in[0,t]}\big(f(s)-z(s)\big)^-.
\end{equation}
This provides, in particular, one of the standard constructions of the reflection of a (drifted) Brownian motion off a continuous function.
On the other hand, it is not hard to see that for fixed $z$ the mapping $f\longmapsto{\cal R}_zf$ is continuous with respect to the topology of uniform convergence in compact subsets of $[0,\infty)$.
And we have, from \eqref{eq:tildeXSA}, that $\tilde\fX^\ep_k={\cal R}_{\tilde\fX^\ep_{k-1}}\!\!\!S^\ep_k$.
Choosing the $q_k$ parameters as $q_k=\tfrac12(1+\ep^{1/2}b_k)$, Donsker's invariance principle implies that each of the scaled random walks $S^\ep_k$ converges to a Brownian motion with diffusivity $2$ and drift $-2b_k$ (all Brownian motions below will implicitly have diffusivity $2$).
Suppose, on the other hand, that $X^\circ_0$ is chosen in such a way that $\fX^\ep_0\longrightarrow\ff_0$ in distribution, uniformly on compact sets, for some continuous (possibly random) function $\ff_0\!:[0,\infty)\longrightarrow\rr$.
We deduce that $\fX^\ep_1(x)$ converges in distribution to a Brownian motion with drift $-2b_1$ reflected off $\ff_0$ and, inductively, that $\fX^\ep_k(x)$ converges in distribution to a Brownian motion with drift $-2b_k$, started at the origin, and reflected off the $(k-1)$-th path.
This is the system of RBMs $(Z_k)_{k=1,\dotsc,m}$ introduced in Section \ref{sec:KPZfluct}, except that now the first path $Z_1$ is reflected off $\ff_0$ instead of the origin; we will denote it as $(Z^{\ff_0}_k)_{k=1,\dotsc,m}$.
We have proved:

\begin{prop}\label{prop:fXeplim}
Assume that $\fX^\ep_0\longrightarrow\ff_0$ in distribution, uniformly on compact sets.
Then $(\fX^\ep_k)_{k=1,\dotsc,m}$ converges in distribution, in the topology of uniform convergence on compact sets, to the system of reflected Brownian motions with drift with a wall at $\ff_0$, $(Z^{\ff_0}_k)_{k=1,\dotsc,m}$.
\end{prop}

\subsection{Scaling limit}

Our next task is to compute the scaling limit of the TASEP formula \eqref{eq:TASEPformula} with initial data as specified in the last section.
We fix $r,x\in\rr$ and choose
\[n=\tfrac12\ep^{-3/2}t-\ep^{-1}x-\tfrac12\ep^{-1/2}a+1,\qquad a=2\ep^{-1}x-2.\]
In view of Proposition \ref{prop:fXeplim}, and since uniform convergence implies convergence in the Hausdorff topology, from Proposition 3.6 or Theorem 3.13 in \cite{Matetski-Quastel-Remenik17} we deduce that the left hand side of \eqref{eq:TASEPformula-o} converges to $\pp_{\cz^{\ff_0}}(\fh(t,x)\leq r)$, where the initial data is chosen as
\begin{equation}\label{eq:czdef}
\cz^{\ff_0}(x)=\begin{dcases*}Z^{\ff_0}_m(-x) & if $x\leq0$\\-\infty & if $x>0$.\end{dcases*}
\end{equation}
The limit of the Fredholm determinant on the right hand side of \eqref{eq:TASEPformula-o} also follows from the arguments \cite{Matetski-Quastel-Remenik17}.
Note that the Fredholm determinant only depends directly on $X^\circ_0$, and not on $X^\circ_k$ for $k=1,\dotsc,m$.
The Fredholm determinant considered in \cite{Matetski-Quastel-Remenik17} is exactly the same one with $m=0$, so it is enough to explain how to hanlde the extra rational factors in the integrand coming from the PushTASEP part.
We do this next.

In taking the limit of \eqref{def:sm} and \eqref{def:sn}, \cite{Matetski-Quastel-Remenik17} uses the change of variables $w\longmapsto\frac12(1-\ep^{1/2}\tilde w)$.
Recall from Proposition \ref{prop:TASEPdTASEP} that the contour $\Gamma_0$ encircles $0$ and $1-q_k$ for each $k$, but not $1$.
Given our choice $q_k=\frac12(1+\ep^{1/2}b_k)$, the new contour after scaling has to encircle $\ep^{-1/2}$ and all $b_k$'s, but not $-\ep^{1/2}$, and hence (for small $\ep$) we may choose it to be a circle $C_\ep$ of radius $\ep^{-1/2}$ centered at $\ep^{-1/2}$, as in \cite{Matetski-Quastel-Remenik17}.
After this change of variables, the pointwise limit of the integrand in \eqref{def:sm} and \eqref{def:sn} is the same as in \cite{Matetski-Quastel-Remenik17} except for the additional factors coming from the rational perturbation in \eqref{eq:psimt}.
Moreover, it can be checked that the steepest descent arguments used in Appendix B of that paper to upgrade this to trace class convergence of the whole operator are not affected by these additional factors; the argument is lengthy but the adaptation is straightforward, so we omit it (the crucial points being, first, that since the additional poles at $\tilde w=b_k$ lie inside the contour $C_\ep$ and all the necessary deformations of it, and second, that the required estimates depend on terms of order $\ep^{-3/2}$ in the exponent after writing the integrand as $e^{F_\ep(w)}$, whereas the rational perturbation is of order $1$).
The upshot is that it is enough to compute the limit of the rational perturbations in $\psi_t(w)$ and $1/\psi_t(1-w)$ after scaling.
To this end we multiply $(\SM_{-t,-n})^*$ by $(2\ep)^{m/2}$ and ${\SN}_{-t,n}^{\epi(X^\circ_0)}$ by $(2\ep)^{-m/2}$ and note that, as $\ep\to0$,
\[2\ep^{1/2}\frac{q}{1-(1-q)\tts w^{-1}}\longrightarrow\frac1{b-\tilde w},\qquad\frac{\ep^{-1/2}}{2}\frac{1-(1-q)\tts(1-w)^{-1}}{q}\longrightarrow\tilde w+b.\]
In view of this and \cite[Lem. 3.5]{Matetski-Quastel-Remenik17} we define the operators
\[\fT_{t,x}^{\vec b,-}(z)=\frac1{2\pi\I}\int_{\langle}\ts dw\,e^{\frac{t}3 w^3+xw^2+zw}\prod_{k=1}^m\frac1{b_k-w},\quad\fT_{t,x}^{\vec b,+}(z)=\frac1{2\pi\I}\int_{\langle}\ts dw\,e^{\frac{t}3 w^3+xw^2+zw}\prod_{k=1}^m(b_k+w),\]
where the contours start at the origin and go off in rays at angles $\pm\pi/3$ with the first one crossing the real axis to the left of all $b_i$'s.
Define also $\fT_{-t,x}^{\vec b,\pm}(z)=\fT_{t,x}^{\vec b,\pm}(-z)$, and let
\begin{equation}
\fT^{\hypo(\ff_0),\vec b,+}_{t,x}(v,u)=\ee_{\fB(0)=v}\big[\fT^{\vec b,+}_{t,x-\tau}(B(\tau),u)\uno{\tau<\infty}\big]\label{eq:hypofh0}
\end{equation}
where $\tau$ is the hitting time of the hypograph of $\ff_0$ by a Brownian motion $B(x)$.
These operators coincide with those defined in \cite[Sec. 3]{Matetski-Quastel-Remenik17} in the case $m=0$, and all the arguments there apply to our case without difference (see \cite{Matetski-Quastel-Remenik20}).
The result is then that, under this scaling, $\det\!\left(I-\bar\P_a(\SM_{-t,-n})^*{\SN}_{-t,n}^{\epi(X^\circ_0)}\bar\P_a\right)$ converges, as $\ep\to0$, to $\det\!\left(\fI-\P_r(\fT^{\vec b,-}_{t,x})^*\fT_{t,-x}^{\hypo(\ff_0),\vec b,+}\P_r\right)$.
We deduce:

\begin{prop}\label{prop:genZform}
For any $t>0$ and $x,r\in\rr$ we have
\eq
\pp\!\left(\fh(t,x;\cz^{\ff_0})\leq r\right)=\det\!\left(\fI-\P_r(\fT^{\vec b,-}_{t,x})^*\fT_{t,-x}^{\hypo(\ff_0),\vec b,+}\P_r\right).\label{eq:fpff0}
\eeq
\end{prop}

As we mentioned in Remark \ref{rem:multi}, everything we have done can be extended without any difficulty to multi-point distributions for $\fh(t,\cdot;\cz^{\ff_0})$, leading to a Fredholm determinant formula involving an extended version of the kernel on the right hand side of \eqref{eq:fpff0}, in the same way as in \cite{Matetski-Quastel-Remenik17}.
The distributional identity \eqref{eq:airywkpzfp} which we stated for the Airy process with wanderers follows directly from this: we need to take $\ff_0$ to be the $\UC$ function which is equal to $0$ at $0$ and $-\infty$ everywhere else, and in that case the hitting time $\tau$ in \eqref{eq:hypofh0} equals $0$ if $v\leq0$ and $\infty$ otherwise, so $\fT^{\hypo(\ff_0),\vec b,+}_{t,x}=\bar\P_0\fT^{\vec b,+}_{1,x}$, and using this in \eqref{eq:fpff0} gives the result (for the multi-point distribution an additional, simple change of variables is needed to turn the heat kernels appearing in the KPZ fixed point formulas into the semigroup $e^{-t\fH}$).

Finally we consider our case of interest, namely half-periodic initial data for the PushTASEP dynamics $X^\circ_0(i)=-2i$, $i\geq1$, which in the scaling \eqref{eq:X0sc} means that $\fX^\ep_0(x)\longrightarrow0$ for $x\leq0$ (and $-\infty$ for $x>0$).
We wish then to compute $\fT_{t,-x}^{\hypo(0),\vec b,+}$, and this can be done explicitly by the reflection principle, since it only involves the passage time below the origin for the Brownian motion $B$.
This was done in \cite[Prop. 3.7]{Quastel-Remenik19} (see also \cite[Sec. 4.4]{Matetski-Quastel-Remenik17}) in the case $m=0$, and the argument extends to general $m$ without difficulty.
It yields $\fT_{t,-x}^{\hypo(0),\vec b,+}=\bar\P_0(\fI+\varrho_0)\fT_{t,-x}^{\vec b,+}$, and thus we get
\begin{equation}\label{eq:Z0form}
\pp\!\left(\fh(t,x;\cz^0)\leq r\right)=\det\!\left(\fI-\P_r(\fT^{\vec b,-}_{t,x})^*\bar\P_0(\fI+\varrho_0)\fT_{t,-x}^{\vec b,+}\P_r\right).
\end{equation}

\subsection{Stationary limit}\label{sec:statlim}

The last step in the proof of Theorem \ref{thm:KPZfp} consists in taking $x\to-\infty$ in \eqref{eq:Z0form}.
We do this in two separate results:

\begin{prop}\label{prop:stat1}
\[\lim_{x\to-\infty}\det\!\left(\fI-\P_r(\fT^{\vec b,-}_{t,x})^*\bar\P_0(\fI+\varrho_0)\fT_{t,-x}^{\vec b,+}\P_r\right)=\det\!\left(\fI-\P_0\tfB^{(\infty,\vec b)}_{t,2^{2/3}r}\P_0\right)\]
where
\begin{equation}\label{eq:Babt}
\tfB^{(\infty,\vec b)}_{t,2^{2/3}r}(z_1,z_2)=\frac1{2\pi\I}\int_{\langle}dw\,e^{2tw^3/3-(z_1+z_2+2r)w}\prod_{k=1}^m\frac{b_k+w}{b_k-w},
\end{equation}
with the contour $\langle$ passing to the left of all the $b_i$'s.
\end{prop}

\begin{prop}\label{prop:stat2}
Let $\big(Z^0_k(t)_{t\geq0}\big)_{k=1,\dotsc,m}$ be the system of RBMs introduced in Section \ref{sec:inilim} (with the first path being reflected off the origin).
Fix $T>0$ and define shifted processes $(\tilde Z^T_k(t))_{t\in\rr}$ by $\tilde Z^T_k(t)=Z^0(t+T)\uno{t+T\geq0}$.
Then the system $(\tilde Z^T_k)_{k=1,\dotsc,m}$ converges in distribution, uniformly on compact sets, to the stationary system of RBMs $(Z^\equ_k)_{k=1,\dotsc,m}$ introduced in Section \ref{sec:KPZfluct}.
\end{prop}

We will actually prove the convergence in a stronger sense (total variation).

By shift invariance of the KPZ fixed point we have $\fh(t,x;\cz^0)\stackrel{\text{dist}}{=}\fh(t,0;\cz^0(x+\cdot))$.
The second result and \eqref{eq:czdef} imply that the shifted process $\cz^0(x+\cdot)$ converges in distribution in $\UC$ as $x\to-\infty$ to the (reversed) process $\varrho_0\cz^{(\vec b)}_\equ$ (recall $\varrho_0f(x)=f(-x)$.
By continuity of the KPZ fixed point transition probabilities with respect to the initial data (in $\UC$, see \cite[Thm. 4.1]{Matetski-Quastel-Remenik17}), we deduce that
\[\fh(t,x;\cz^0)\xrightarrow[x\to-\infty]{}\fh\big(t,0;\varrho_0\cz^{(\vec b)}_\equ\big)\]
in distribution.
But the KPZ fixed point is reflection invariant, $\fh(t,x;\varrho_0\fh_0)\stackrel{\text{dist}}{=}\fh(t,-x,\fh_0)$ (see \cite[Thm. 4.5]{Matetski-Quastel-Remenik17}), so this together with \eqref{eq:Z0form} and Proposition \ref{prop:stat1} yields
\begin{equation}\label{eq:Zequform}
\pp\!\left(\fh(t,0;\cz^{(\vec b)}_\equ)\leq r\right)=\det\!\left(\fI-\P_0\tfB^{(\infty,\vec b)}_{t,2^{2/3}r}\P_0\right).
\end{equation}
Setting $t=1$, and in view of Theorem \ref{thm:DetFormula}, this yields Theorem \ref{thm:KPZfp}.

We turn then to the proof of the two propositions.

\begin{proof}[Proof of Proposition \ref{prop:stat1}]
In order to prove convergence of the Fredholm determinant on the right hand side of \eqref{eq:Z0form} we need to show that the kernel inside it converges in trace norm in $L^2([r,\infty))$.
The arguments which achieve this are relatively standard, so we will skip some details.
For simplicity we set $t=1$; the general case follows by scaling.

We may multiply the kernel by $e^{(z_1-z_2)x}$ without changing the value the Fredholm determinant, since this is just a conjugation.
The resulting kernel then equals $\fL_1+\fL_2$ with
\begin{equation}\label{eq:ker}
\fL_1(z_1,z_2)=e^{(z_1-z_2)x}(\fT^{\vec b,+}_{1,-x})^*\bar\P_0\fT^{\vec b,-}_{1,x}(z_1,z_2),\quad
\fL_2(z_1,z_2)=e^{(z_1-z_2)x}(\fT^{\vec b,+}_{1,-x})^*\bar\P_0\varrho_0\fT^{\vec b,-}_{1,x}(z_1,z_2).
\end{equation}
When $m=0$, $\fL_1+\fL_2$ is the kernel whose Fredholm determinant computes the one-point distribution (at $x$) of the Airy$_{2\to1}$ process, see \cite[Sec. 4.4]{Matetski-Quastel-Remenik17}.
In that case it is known that, as $x\to-\infty$, and in trace norm in $L^2([r,\infty))$, $\fL_1$ goes to $0$ and $\fL_2$ goes to $\fB_{2^{2/3}r}$ (see \cite{Borodin-Ferrari-Sasamoto08a,Quastel-Remenik13}).
In order to generalize this to our case we need to handle the rational perturbations in our kernels.
The proof for the Airy$_{2\to1}$ case is based on the bound $|\!\Ai(z)|\leq C\tts e^{-\frac23(z\vee0)^{3/2}}$ for the Airy function.
The key to the extension is the following estimate: for $z>0$ and any $b_1,\dotsc,b_m\in\rr$, if $\langle'$ denotes the usual Airy contour $\langle$ but shifted so that it crosses the real axis at $\sqrt{z}$, we have
\begin{equation}\label{eq:airyGestimate}
\left|\frac1{2\pi\I}\int_{\langle'}dw\,e^{w^3/3-zw}\prod_{k=1}^m(b_k\pm w)^{\pm 1}\right|
\leq C\tts e^{-\frac23z^{3/2}}\prod_{k=1}^m(b_k+\sqrt{z})^{\pm 1}
\end{equation}
for some $C>0$; this can be proved using Laplace's method and the method of steepest descent in exactly the same way as in the classical estimate for the Airy function (see e.g. \cite{Stein-Shakarchi03}), which corresponds to $m=0$.

Consider first $\fL_1$.
We want to prove that this operator goes to $0$ in trace norm (in $L^2([r,\infty))$) as $x\to-\infty$.
We have
\begin{equation}
\fL_1(z_1,z_2)
=\frac1{(2\pi\I)^2}\int_{-\infty}^0d\eta\int_{\langle}\ts dw\int_{\langle}\ts dv\,e^{\frac{1}{3}w^3-x w^2-(z_1-\eta)w+\frac{1}{3}v^3+x v^2-(z_2-\eta)v+(z_1-z_2)x}\prod_{k=1}^m\frac{b_k+w}{b_k-v}.\label{eq:piece1}
\end{equation}
The $v$ contour passes to the left of all $b_i$'s (the $w$ contour does not really have a restriction since the integrand is analytic in $w$), but in taking $x\to-\infty$ it will be convenient to have it lie to the right of these points.
So we shift the contour in this way, collecting the residues coming from the poles in the rational factor in the integrand.
If all $b_i$'s are different then we pick up $m$ residues, the $i$-th one being (here we choose $\langle$ so that $\Re(w+b_i)>0$ for each $i$)
\begin{multline}
\frac1{2\pi\I}\int_{-\infty}^0d\eta\int_{\langle}\ts dw\,e^{\frac{1}{3}w^3-x w^2-(z_1-\eta)w+\frac{1}{3}b_i^3+xb_i^2-(z_2-\eta)b_i+(z_1-z_2)x}\frac{\prod_k(b_k+w)}{\prod_{k\neq i}(b_k-b_i)}\\
=\frac1{2\pi\I}\int_{\langle}\ts dw\,e^{\frac{1}{3}w^3-x w^2-z_1w+\frac{1}{3}b_i^3+xb_i^2-z_2b_i+(z_1-z_2)x}\prod_{k\neq i}\frac{b_k+w}{b_k-b_i},
\end{multline}
and now changing variables $w\longmapsto w+x$ we get
\[\frac{e^{-\frac23x^3+\frac{1}{3}b_i^3+xb_i^2-z_2(b_i+x)}}{2\pi\I}\int_{\langle}\ts dw\,e^{\frac{1}{3}w^3-(z_1+x^2)w}\prod_{k\neq i}\frac{b_k+w+x}{b_k-b_i}
.\]
This is a rank-1 kernel, and thus its trace norm is just the product of the $L^2$ norms in $z_1$ and $z_2$ (over $[r,\infty)$), which goes to $0$ as $x\to-\infty$ thanks to \eqref{eq:airyGestimate} (note that $\langle$ can be shifted to cross the real axis at $\sqrt{z_1+x^2}$ as needed without trouble, and that the prefactor $e^{-2x^3/3}$ gets canceled precisely by the asymptotics coming from \eqref{eq:airyGestimate}).
If some of the $b_i$'s coincide then we have higher order poles, but the residues look the same as the ones we just analyzed except with an extra polynomial in the $b_i$'s, $x$ and $z_2$ as a prefactor, which does not affect the argument.
We are thus left with estimating the right hand side of \eqref{eq:piece1} after having moved the $v$ contour.
Changing variables $w\longmapsto w+x$, $v\longmapsto v-x$, we get
\begin{equation}
\frac1{(2\pi\I)^2}\int_{-\infty}^0d\eta\int_{\langle}\ts dw\int_{\langle}\ts dv\,e^{\frac{1}{3}w^3-(z_1-\eta+x^2)w+\frac{1}{3}v^3-(z_2-\eta+x^2)v}\prod_{k=1}^m\frac{b_k+x+w}{b_k+x-v},
\end{equation}
where the $v$ contour now passes to the right of $b_i+x$ for all $i$, and can thus remain fixed as $x\to-\infty$.
We express this as the product of two kernels,
\begin{align}
\fR_1(z_1,\eta)&=\frac1{2\pi\I}\int_{\langle}\ts dw\,e^{\frac{1}{3}w^3-(z_1-\eta+x^2)w}\prod_{k=1}^m(b_k+x+w)\uno{\eta\leq0},\\
\fR_2(\eta,z_2)&=\frac1{2\pi\I}\int_{\langle}\ts dv\,e^{\frac{1}{3}v^3-(z_2-\eta+x^2)v}\prod_{k=1}^m\frac{1}{b_k+x-v}\uno{\eta\leq0},
\end{align}
so that the trace norm which we are interested is bounded by the product of Hilbert-Schmidt norms $\|\P_r\fR_1\|_2\|\fR_2\P_r\|_2$; both norms can be estimated using \eqref{eq:airyGestimate} for large enough $x$, yielding a bound which decays like $e^{-c\tts|x|^{3/2}}$ times a polynomial in $x$.

We have shown then that $\P_r\fL_1\P_r$ goes to $0$ in trace norm as $x\to-\infty$.
Consider now the kernel $\fL_2$, which we decompose as $\fL_{2,1}-\fL_{2,2}$ with
\begin{equation}\label{eq:s3}
\fL_{2,1}(z_1,z_2)=e^{(z_1-z_2)x}(\fT^{\vec b,+}_{1,-x})^*\varrho_0\tts\fT^{\vec b,-}_{1,x}(z_1,z_2),\quad
\fL_{2,2}(z_1,z_2)=e^{(z_1-z_2)x}(\fT^{\vec b,+}_{1,-x})^*\P_0\varrho_0\tts\fT^{\vec b,-}_{1,x}(z_1,z_2).
\end{equation}
In the case $m=0$ the term $\fL_{2,2}$ goes to $0$ in trace class in $L^2([r,\infty))$; although it is a bit more complicated, the argument can be adapted to handle the rational perturbations in the case $m\geq1$ in a similar way as we did above for $\fL_1$ (see the comment after \cite[Eqn. (2.16)]{Quastel-Remenik13}).
We are then only left with $\fL_{2,1}$, which equals
\[\frac1{(2\pi\I)^2}\int_{-\infty}^\infty d\eta\int_{\langle}\ts dw\int_{\langle}\ts dv\,e^{\frac{1}{3}w^3-x w^2-(z_1-\eta)w+\frac{1}{3}v^3+x v^2-(z_2+\eta)v+(z_1-z_2)x}\prod_{k=1}^m\frac{b_k+w}{b_k-v}.\]
Recall that the $v$ contour lies to the left of all $b_i$'s, but the $w$ contour is free.
Then, proceeding similarly to \eqref{eq:wtB1}--\eqref{eq:wtB2}, we may deform both contours  to $\I\rr-c$ for some large $c$, so that the $\eta$ integral yields $\delta_0(u-v)$ and, deforming back the contour, we get
\[\fL_1(z_1,z_2)=\frac1{2\pi\I}\int_{-\infty}^\infty d\eta\int_{\langle}\ts dw\,e^{\frac{2}{3}w^3-(z_1+z_2)w+(z_1-z_2)x}\prod_{k=1}^m\frac{b_k+w}{b_k-w}.\]
Note that this kernel does not depend on $x$ except for the conjugation $e^{(z_1-z_2)x}$; removing it and shifting variables $z_1\longmapsto z_1+r$, $z_2\longmapsto z_2+r$ in the Fredholm determinant yields the result.
\end{proof}

\begin{proof}[Proof of Proposition \ref{prop:stat2}]
Consider the gap process $Y_k=Z^0_k-Z^0_{k-1}$ associated to our system of RBMs $(Z^0_k)_{k=1,\dotsc,m}$.
From \cite{Harrison-Williams87} (see also \cite{Williams95}) it is known that $Y=(Y_1,\dotsc,Y_m)$, has a unique invariant distribution, call it $\pi$.
Moreover, if ${\cal L}_t$ denotes the law of $Y(t)$ then one has
\[\big\|{\cal L}_t-\pi\big\|_{\rm TV}\xrightarrow[t\to\infty]{}0\]
where $\|\cdot\|_{\rm TV}$ denotes the total variation norm, see \cite{Dupuis-Williams94} (even stronger results are proved in \cite{Budhiraja-Lee07,Sarantsev17a}).
The rest of the proof follows from abstract arguments, which we describe next.

For $T>0$ define the shifted gap process $(Y^T(t))_{t\in\rr}$ as $Y^T(t)=Y(t+T)\uno{t+T\geq0}$, and define also $Y^\equ$ to be the double-sided stationary version of $Y$ (with $\pi$ as its marginal).
We will prove that, for fixed $L>0$, $Y^T\longrightarrow Y^\equ$ in $[-L,L]$ as $T\to\infty$ in distribution in total variation distance as random variables in $\mathcal{C}([-L,L],\rr^m)$, the space of continuous $m$-dimensional paths on $[-L,L]$ with the supremum norm.
To this end we may define a measurable function $G$ from $\rr\times[0,1]$ to $\mathcal{C}([-L,L],\rr^m)$ so that if $U$ is a uniform random variable on $[0,1]$ then $G(y,U)$ has the distribution of the gap process $Y$ on $[0,2L]$, started from $y$ (see e.g. \cite[Lem. 3.22]{Kallenberg02}).
Consider, on the other hand, the optimal coupling between ${\cal L}_T$ and $\pi$ (which realizes their total variation distance) and use it to draw a random pair $(\Gamma_1,\Gamma_2)$.
Define, additionally, a uniform random variable $U$ on $[0,1]$, which is also independent of $(\Gamma_1,\Gamma_2)$, and let $Y_{(i)}=G(\Gamma_i,U)$.
Then $Y_{(1)}\stackrel{\text{dist}}{=}Y^T$ and $Y_{(2)}\stackrel{\text{dist}}{=}Y^\equ$ (on $[-L,L]$) and the total variation distance between the laws of $Y_{(1)}$ and $Y_{(2)}$ on $[-L,L]$ is bounded by $\|{\cal L}_{T-L}-\pi\|_{\rm TV}$, which goes to $0$ as $T\to\infty$.

From this we get directly the convergence of $\tilde Z^T$ to $Z^\equ$ on $[-L,L]$ in total variation norm, which implies the result.
\end{proof}

\subsection{Proof of the limit transitions}\label{sec:limittrans}

Our goal here is to prove \eqref{eq:limits1} and \eqref{eq:limits2}.
By \eqref{eq:absymm} it is enough to consider the case where all $a_i$'s equal $\infty$.

Consider first the second limit.
It is easy to couple $\cz^{(\vec b)}_\equ$ (with general $m$) and $\cz^{(b_1)}_\equ$ (with $m=1$), e.g. using the Skorokhod construction, in such a way that $\cz^{(b_1)}_\equ(x)\leq\cz^{(\vec b)}_\equ(x)$ for all $x$ (after all, $\cz^{(b_1)}_\equ$ is essentially the lower path in the stationary system with drifts $\vec b$).
This implies by \cite[Prop. 4.5]{Matetski-Quastel-Remenik17} that $\fh(1,x;\cz^{(b_1)}_\equ)$ is stochastically dominated by $\fh(1,x;\cz^{(\vec b)}_\equ)$, and thus the statement for $m=1$, $\cz^{(b_1)}_\equ\longrightarrow\infty$ in distribution as $b_1\to0$, yields \eqref{eq:limits2}.

Consider next \eqref{eq:limits1}.
We can construct the stationary system $Z^\equ$ on any interval $[-L,L]$ using the Skorokhod construction with initial data $Z^\equ(-L)$ distributed according to the stationary distribution of the system.
Recall that $Z^\equ_1(-L)$ is an exponential random variable with parameter $-2b_1$, and it thus converges to $0$ in distribution as $b_1\to\infty$.
Using the recursive construction one checks now that the remaining components of the initial data $Z^\equ(-L)$ also go to $0$ in distribution.
In the same way, from the Skorokhod construction it is easy to show that the lower path $Z^\equ_{1}$ goes to $0$ in distribution in $[-L,L]$ as $b_1\to\infty$, and then recursively that each path $Z^\equ_k$, $k=1,\dotsm,m$ satisfies the same as all drifts go to $\infty$.
This shows that $\cz^{(\vec b)}_\equ$ goes to $0$ in distribution, uniformly on compact sets, when $b_1,\dotsc,b_m\to\infty$.
Since the KPZ fixed point transition probabilities are continuous in the initial data under this convergence, we get $\fh(1,0;\cz^{(\vec b)}_\equ)\longrightarrow\fh(1,0;0)$ in distribution as all drifts go to infinity, and the result follows since $\pp(\fh(1,0;0)\leq r)=F_1(2^{2/3}r)$.

\appendix

\section{Brownian Excursions and discrete orthogonal polynomials}\label{BEandOP}

In this appendix we sketch an alternative proof of Theorem \ref{thm:PII} which leads rather directly to \eqref{eq:PII1} in the cases $m=1,2$, but becomes increasingly difficult for larger values of $m$.  The model we consider is as follows. We consider $N$ Brownian excursions $X_1(t), \dots, X_N(t)$, i.e., Brownian bridges on $\rr_+$ with an absorbing wall at zero, which are conditioned to begin at position 0 at time $t=0$. We also condition on the particles ending points at time $t=1$ to be $X_1(1) = X_{2}(1) = \dots = X_{N-m}(1) = 0$, and $X_N(1) = \beta_1, \dots, X_{N-m+1}(1)= \beta_m$, where $\beta_1\ge \beta_2 \ge \dots \ge \beta_m\ge 0$. The particles are conditioned not to intersect at times $0<t<1$, and are ordered as $X_1(t) <X_2 (t) < \dots < X_N(t)$. Introduce the notation
\[
\mcal_N := \max_{t\in (0,1)} X_N(t)
\]
and consider the limit
\eq
  \lim_{N\to\infty}P(\mcal_N < h )
  \eeq
with the scalings
 \begin{equation}\label{scaling}
 \beta_j=\sqrt{2N}-(2N)^{1/6} b_j \quad \textrm{for} \quad 1\le j \le r ,\qquad h=\sqrt{2N}+\frac{r}{2(2N)^{1/6}}, \qquad r\in \rr, \quad b_j> 0.
 \end{equation}
    
For the case $m=0$, i.e. for non-intersecting Brownian excursions without outliers, a similar scaling limit was proven in \cite{Liechty12} using the asymptotic analysis of a system discrete orthogonal polynomials via the Riemann--Hilbert method. Here we use the same approach, although the presence of outliers makes the analysis somewhat more involved.

Using the Karlin--McGregor formula for non-intersecting paths \cite{Karlin-McGregor59} along with the formulas relating expected values of products of characteristic polynomials in random matrix models to orthogonal polynomials, see \cite[Thm. 2.3]{Baik-Deift-Strahov03},  \cite[Eqn. (14)]{Brezin-Hikami00}, we arrive at the following exact formula for the distribution function of $\mcal_N$:
\begin{equation}\label{r3b}
\begin{aligned}
P(\mcal_N < h )
&= \frac{ (-1)^{m(N-m)} 2^m \pi^{N^2+(N-m)^2+N/2}}{2^{N/2} h^{N^2+(N-m)^2+N}\De_m({\bf \beta}^2)} \prod_{j=1}^m \frac{e^{\frac{1}{2}\beta_j^2}(2N-2j+1)!}{\beta_j^{2(N-m)+1}}\prod_{j=0}^{N-1} \frac{\tilde{h}_{2j+1}}{(2j+1)!}  \\
&\qquad\times
 \det\!\left[ \sum_{x_k=1}^\infty  \sin\left(\frac{x_k \pi \beta_k}{h}\right)  e^{-\frac{\pi^2}{2h^2} x_k^2}\frac{P_{2(N-m+j)+1}(x_{k})}{\tilde{h}_{2(N-m+j)-1}}\right]_{j,k=1}^m, \\
\end{aligned}
\end{equation}
where $\De(\beta^2)= \prod_{1\le j<k\le m}(\beta_k^2-\beta_j^2)$ and where $P_k(x)$ are the monic polynomials of degree $k$, defined via the orthogonality condition
\begin{equation}\label{def:OPs}
\sum_{x=-\infty}^\infty P_k(x)P_j(x) e^{-\frac{\pi^2}{2h^2} x2} = {h}_k \de_{jk}.
\end{equation}
With the scalings \eqref{scaling}, in which case it is known that
\begin{equation}
 \lim_{N\to\infty} \frac{\pi^{2N^2+N/2}}{2^{N/2} h^{N(2N+1)}}\prod_{j=0}^{N-1} \frac{{h}_{2j+1}}{(2j+1)!} = F_1(2^{2/3}r)
\end{equation}
(see \cite[Eqn. (1.20) and Thm. 1.1]{Liechty12}), we get from \eqref{r3b} that, for large $N$,
\begin{align}\label{r3c}
P(\mcal_N < h )  &\sim (-1)^{m(N-m)+1}\frac{F_1(2^{2/3} r)}{\De_m({\bf \beta}^2)}\left(\frac{h}{\pi}\right)^{m(2N-m)} \\ &\times \det\left[\frac{e^{\frac{1}{2}\beta_k^2}(2N-2j+1)!}{\beta_k^{2(N-m)+1}}  \sum_{x_k=-\infty}^\infty  \sin\left(\frac{x_k \pi \beta_k}{h}\right)  e^{-\frac{\pi^2}{2h^2} x_k^2}\frac{P_{2N-2j+1}(x_{k})}{{h}_{2N-2j+1}}\right]_{j,k=1}^m. \\
\end{align}
The sums in the determinant of \eqref{r3c} can then be computed entrywise as $N\to\infty$. To do so one can use the asymptotic formulas for the orthogonal polynomials which follow from the Riemann--Hilbert analysis in \cite{Liechty12} (see also \cite{Buckingham-Liechty19,Liechty-Wang17}), rewrite the sum as a contour integral, and perform classical steepest descent analysis following the approach of \cite[Sec. 5.2]{Liechty-Wang17}. Since each row of the matrix in \eqref{r3c} has the same leading-order behavior as $N\to\infty$, one needs to use the subleading terms which are the difference between two rows to compute the limit. For $m=1,2$, this leads to the RHS of \eqref{eq:PII1}. For $m>2$ a similar approach would work, but would require  knowledge of more terms in the asymptotic expansion of the orthogonal polynomials \eqref{def:OPs} than we were willing to compute. If there are $m$ wanderers, the computation would require knowledge of $m-1$ subleading terms in the large $N$ asymptotic expansion of the the orthogonal polynomials. For any fixed $m$, this is doable by the Riemann--Hilbert method, but becomes increasingly difficult as $m$ becomes larger.

We also remark that the choice to use Brownian excursions rather than standard Brownian bridges like \cite{Adler-Ferrari-van_Moerbeke10} is only to match the similar analysis in \cite{Liechty12}. We could have taken a similar approach starting from non-intersecting Brownian bridges with wanderers. In this case the relevant orthogonal polynomials are continuous ones defined on a half-line $(-\infty ,z_0)$, and the critical phenomenon occurs when the edge of the oscillatory region for the orthogonal polynomials is close to the ``hard edge" at $z_0$. This phenomenon has been studied in several places, and is also related to the Painlev\'{e} II (and the Painlev\'{e} XXXIV) equation \cite{Claeys-Kuijlaars-vanLessen08, Its-Kuijlaars-Ostensson08, Nadal-Majumdar11,Perret-Schehr14, Xu-Zhao19}.

\vskip12pt
\paragraph*{Acknowledgments.}

The authors thank Mauricio Duarte for background and references on RBMs and Joaqu\'in Fontbona for helping us with the abstract argument which proves the convergence to a stationary process for the system of RBMs with drift. 

\noindent Throughout this project we made extensive use of Folkmar Bornemann's MATLAB package for numerical computation of Fredholm determinants \cite{Bornemann10a}, as well as the Mathematica package \texttt{RHPackage} by Sheehan Olver for numerical solutions to Riemann--Hilbert problems \cite{Olver12}. 

\noindent KL was supported by a Simons Foundation Collaboration Grant \#357872. GBN was supported by the Swedish Research Council Grant 67465 VR20BN. DR was supported by CMM ANID Grant AFB170001, by Programa Iniciativa Cient\'ifica Milenio grant number NC120062 through Nucleus Millenium Stochastic Models of Complex and Disordered Systems, and by Fondecyt Grant 1201914.

\bibliographystyle{plain}
\bibliography{bibliography.bib}
\addcontentsline{toc}{section}{References}

\end{document}